\documentclass[a4paper,12pt]{article}
\usepackage{amsmath, amsthm, amssymb, fullpage}
\usepackage{parskip}
\usepackage[english]{babel}
\usepackage[utf8]{inputenc}
\usepackage{eucal}
\usepackage{mathtools}
\usepackage[colorlinks, bookmarks=true]{hyperref}
\hypersetup{
	colorlinks,
	citecolor=blue,
	linkcolor=blue,
	urlcolor=blue}
\usepackage{cleveref}
\usepackage{titlesec}
\usepackage{xcolor}
\usepackage{indentfirst}
\usepackage[nottoc]{tocbibind}
\usepackage[noadjust]{cite}
\usepackage{graphicx}
\usepackage{subcaption}

\numberwithin{equation}{section}

\newtheorem{theorem}{Theorem}[subsection]
\newtheorem{cor}[theorem]{Corollary}
\newtheorem{lem}[theorem]{Lemma}
\newtheorem{prop}[theorem]{Proposition}
\newtheorem{rem}[theorem]{Remark}
\newtheorem{defi}[theorem]{Definition}
\newtheorem{conj}[theorem]{Conjecture}
\newtheorem{question}[theorem]{Question}

\newcommand{\bo}{\mathcal{O}}
\newcommand{\R}{\mathbb{R}}
\newcommand{\C}{\mathcal{C}}
\newcommand{\Hy}{\mathbb{H}^3}
\newcommand{\Z}{\mathbb{Z}}
\newcommand{\T}{\mathcal{T}}
\newcommand{\ds}{\displaystyle}
\newcommand{\lh}{\mathcal{H}^1}

\newcommand{\Ima}{\text{Im}(H_2(M) \rightarrow H_2(M, \partial M))}
\newcommand{\fma}{\text{Im}(H^1_0(M) \rightarrow H^1(M))}

\newcommand{\inj}{\textnormal{inj}}
\newcommand{\vol}{\textnormal{vol}}
\newcommand{\sys}{\textnormal{sys}}
\newcommand{\mt}{M_{\tau}}
\newcommand{\tn}{\textnormal}
\newcommand{\ok}{\Omega^k}

\newcommand{\okl}{\Omega^{k-1}}
\newcommand{\okc}{\Omega^k_0}
\newcommand{\okhc}{\Omega^{k+1}_0}

\newcommand{\pc}{Poincar\'e }

\newcommand{\suchthat}{\,|\,}

\title{HARMONIC FORMS, MINIMAL SURFACES AND NORMS ON COHOMOLOGY OF HYPERBOLIC $3$-MANIFOLDS}
\author{Xiaolong Hans Han}
\date{}

\begin{document}
\maketitle

\begin{abstract}
We bound the $L^2$-norm of an $L^2$ harmonic $1$-form in an orientable cusped hyperbolic $3$-manifold $M$ by its topological complexity, measured by the Thurston norm, up to a constant depending on $M$. This generalizes two inequalities of Brock and Dunfield. We also study the sharpness of the inequalities in the closed and cusped cases, using the interaction between minimal surfaces and harmonic forms. We unify various results by defining two functionals on closed and cusped hyperbolic $3$-manifolds and formulate several questions and conjectures.
\end{abstract}

\tableofcontents

\section{Introduction}
\subsection{Motivation and previous results}
A cusped manifold is a complete noncompact hyperbolic manifold with finite volume. All manifolds are assumed oriented unless stated otherwise. By Mostow rigidity, the topology of a closed or cusped hyperbolic manifold $M$ of dimension at least $3$ determines its geometry. Effective geometrization attempts to seek qualitative and quantitative connections between the topological and geometric invariants of the manifold. For example, given the fundamental group, what can we say about the injectivity radius, diameter, or $2$-systole of the manifold (see \cite{injFundGroupWhite, finiteEigenInjIanJuan, largeEBballHeeGenusWhite, systole2Hyperbolic3})? From a conjecture of Bergeron and Venkatesh \cite{bvAsymptoticTorsionArith} and similar conjectures proposed independently by L\"{u}ck and by L\^{e} in \cite{leConjecture, luckSurveyTorsion}, one can extract the volume of a closed (arithmetic) hyperbolic $3$-manifold by the following: 
\begin{conj}\tn{(\cite{bvAsymptoticTorsionArith})}
	Let $M_n$ be a sequence of congruence covers of a fixed arithmetic hyperbolic $3$-manifold $M$. Denote the degree of the cover by $[\pi_1 M: \pi_1 M_n]$. Then we have 
	\begin{equation}
		\ds \tn{vol}(M)=6\pi \lim_{n\rightarrow \infty} \frac{\log |H_1(M_n)_{tor}|}{[\pi_1 M: \pi_1 M_n]},
	\end{equation}
where $H_1(\cdot)$ is the first homology and $(\cdot)_{tor}$ is the torsion part. 
\end{conj}
For recent progress and other related works on this conjecture, see \cite{leHomoTorGrowthMeasure,luckW,bdInjectIntegerHomo}. In \cite{bsvTorsionHomologyArith}, Bergeron, \c{S}eng\"un and Venkatesh make another conjecture that $H_2(M_n)$ can be generated by integral cycles in $H_2(M_n, \R)$ with low topological complexity, as measured by the Thurston norm $\|\cdot\|_{Th}$ (see \Cref{topologyL^2} for detailed definitions). On the other hand, in a general closed Riemannian manifold $M$, one can measure the complexity of a cycle $\phi \in H^1(M; \R)$ by the geometric $L^2$-norm: 
\begin{center}
	$\ds \|\phi\|_{L^2}^2= \int_M \alpha \wedge *\alpha$, 
\end{center}
where $\alpha$ is the harmonic representative of $\phi$ given by the classical Hodge theory. The $L^2$-norms and harmonic forms also enter the picture via \textit{regulators} in \cite{bvAsymptoticTorHomo} (Calegari and Venkatesh [\citenum{cvTorsionLanglands}, 5.3] also generalize the regulator and the analytic torsion to the noncompact setting). The $H_j$-\textit{regulator} of $M$ is the volume of $H_j(M, \R)$ modulo the free part of $H_j(M, \Z)$ with respect to the metric defined by harmonic forms. That is, 
\[
\ds R_j(M)=\frac{\det(\int_{\gamma_p} \alpha_q)}{\sqrt{\det \langle \alpha_p, \alpha_q \rangle}},
\]
where the $\gamma_p \in H_j(M, \Z)$ project to a basis for $ H_j(M, \Z)/ H_j(M, \Z)_{tor}$ and the $\alpha_p$ form a basis for the $L^2$ harmonic $j$-forms on $M$. The celebrated theorems of Cheeger and M\"uller ([\citenum{cjAnalyticTorHeat}, Theorem 8.22], [\citenum{mwAnalyticTorRtor}, Theorem 10.22]) imply that in dimension $3$, 
\[
\ds |H_1(M, \Z)_{tor}| \cdot \frac{R_2^2(M)}{\vol(M)}=\frac{1}{T_{an}(M)},
\]
where $T_{an}(M)$ is the analytic torsion of $M$. In proving that the $H_2$-regulator can be controlled by betti number $b(\cdot)$, $R_2(M_n) <\!\!< \vol(M_n)^{Cb(M_n)}$ ([\citenum{bsvTorsionHomologyArith}, 2.5]), Bergeron, \c{S}eng\"un and Venkatesh proved a beautiful theorem regarding the Thurston norms and the $L^2$-norms: 
\begin{theorem}\tn{(\citenum{bsvTorsionHomologyArith}, Theorem 4.1)}
	If $M$ varies through a sequence of finite coverings of a fixed closed orientable hyperbolic $3$-manifold $M_0$, then we have 
	\begin{equation}
		\frac{C_1}{\vol(M)} \|\cdot\|_{Th} \leq \|\cdot \|_{L^2}\leq C_2\|\cdot\|_{Th} \text{ on } H^1(M;\R),
	\end{equation}
where $C_1$ and $C_2$ depend only on $M_0$. 
\end{theorem}
The above theorem is another example of effective geometrization, where one bounds the geometric complexity of $\phi \in H^1$ by its topological complexity. Brock and Dunfield in [\citenum{bdNorms}] remove the requirement that $M_n$ be covers of a fixed closed manifold $M_0$ and obtain the following inequalities:  
\begin{theorem}\textnormal{[\citenum{bdNorms}, Theorem 1.2]}
	For all closed orientable hyperbolic $3$-manifolds $M$ one has
	\begin{equation}\label{bdi}
		\ds \frac{\pi}{\sqrt{\vol(M)}} \| \cdot \|_{Th} \leq \| \cdot \|_{L^2} \leq \frac{10\pi}{\sqrt{\inj(M)}} \|\cdot \|_{Th} \text{ on } H^1(M;\R).
	\end{equation}
\end{theorem}
Compared to \cite{bsvTorsionHomologyArith}, Brock and Dunfield introduce the tools of minimal surfaces and harmonic expansions of harmonic functions in $\Hy$ for such inequalities. In particular, the proof of the left-hand inequality uses the uniform bounds on the principal curvatures for a stable minimal surface in closed hyperbolic $3$-manifolds and the Gauss-Bonnet theorem. By observation of Uhlenbeck \cite{closedminSurfaceHyperbolic,hass1995acy}, the Thurston norm is uniformly equivalent to the least area norm, which measures the smallest area of a surface representative dual to a cohomology class. Brock and Dunfield then point out the coincidence between the least area norm and the $L^1$-norm and use the Cauchy-Schwarz inequality to estimate the $L^2$-norm by the $L^1$-norm. We note that the left-hand inequality of \eqref{bdi} has been studied by very different techniques. Kronheimer and Mrowka \cite{kronheimer1997scalar} and Lin \cite{linMonopoleSpectral} use estimates on the solutions of the Seiberg-Witten equations. Bray and Stern \cite{bsScalarBoundary} use the harmonic level set techniques, the Bochner technique, and the traced Gauss equations. The proof of the right-hand inequality combines several ideas. It uses explicit computations for the series expansions of harmonic functions and an inequality bounding $L^{\infty}$-norm of a $1$-form on a ball $B$ by its $L^2$-norm on $B$, reminiscent of the mean value property of harmonic functions. Di Cerbo and Stern have done similar $L^\infty$ estimates in \cite{dicerbo2019harmonic,dicerbo2019price,dicerbo2020betti} for real and complex cusped manifolds with dimension $\geq 3$ with implicit constants. 

\subsection{A glimpse at the noncompact case}
It is natural to wonder whether a version of  \eqref{bdi} still holds when $M$ is noncompact. It turns out that the $L^2$-norm can depend on the geometry near infinity in a subtle way, which is not captured by topology. In the cusped case, several problems arise: 
\begin{enumerate}
	\item One immediate challenge is that the injectivity radius of $M$ is zero, and the right-hand inequality of  \eqref{bdi} is trivial. We have to find suitable geometric invariants to replace the injectivity radius. After we develop a deeper understanding of Hodge theory and minimal surfaces on noncompact hyperbolic manifolds, we see that the geometric invariants are the systole of $M$ and the intrinsic diameters of maximal cusp neighborhoods.
	
	\item Originally,  \eqref{bdi} holds in the setting of closed manifolds, where $H^1 \cong H_2$. Brock and Dunfield use the stability of least area incompressible surfaces in closed hyperbolic $3$-manifolds heavily, where one has uniform control on the principal curvatures on such surfaces and the area is bounded from above and below by some constant multiple of the Euler characteristic. For $H_2(M)$ in the noncompact case, there exist closed surfaces that escape to infinity through homotopy supported in the cusps, and the uniform control on curvature is no longer as clear. In addition, the Thurston norm $\|\cdot \|_{Th}$ is no longer a genuine norm on $H_2(M)$ by the following example. Let $T$ be a boundary-parallel torus. If $M$ has more than one cusp, then $[T]\neq 0$ in $H_2(M)$ but $\|T\|_{Th} =0$. In this case, one questions the validity and implication of such inequality. Even for a closed minimal surface $S$ with genus at least $2$ supported in the thick part, a priori, there is no obstruction for a portion of $S$ to enter the cusp arbitrarily deep. 
	
	\item On the other hand, if we look at $\phi \in H^1(M)$, $\| \phi \|_{L^2}$ may blow up. For $M$ closed, by classical Hodge theory, each $\phi \in H^1$ has a harmonic representative, naturally with finite $L^2$-norm. When $M$ is noncompact, there exists $\phi \in H^1(M)$ such that $\| \phi \|_{L^2}= \infty$: through a relatively straightforward computation in \Cref{L2normblowupsurlem}, we demonstrate that for an incompressible properly embedded noncompact surface like the thrice-punctured sphere, the dual $1$-form $\alpha$ has infinite $L^2$-norm. 

\end{enumerate}
\subsection{Main theorem}
We resolve all the issues mentioned above by imposing one condition on $H^1(M)$: 
\begin{center}
	$\psi \in H^1(M)$: $\psi$ is represented by an $L^2$ harmonic $1$-form $\alpha$.
\end{center}

Denote by $\lh$ the space of $L^2$ harmonic $1$-forms on $M$. In \cite{zucker1982} (see also (1.4) and comments before that in \cite{mpHodgeHyp}), Zucker showed that $\lh$ is isomorphic to the image of the natural inclusion of the compactly supported first cohomology $\text{Im}(H^1_0(M) \rightarrow H^1(M))$. Here, we provide a self-contained introduction to Hodge theory for noncompact hyperbolic manifolds in Section \ref{hodgeL^2}, based on \cite{mpHodgeHyp}. We then analyze the topology of the dual surfaces of $\psi$ in Section \ref{topologyL^2} and show that the Thurston norm restricted to $\lh$ is a genuine norm in \Cref{genuinenormlem}. Geometrically, we point out the beautiful coincidence that the surfaces dual to elements of $\lh$ are exactly those described recently in [\citenum{hwClosedMinCusped}, Corollary 1.2]: 
\begin{cor}
	Let $S$ be a closed orientable embedded surface in a cusped $3$-manifold $M^3$ which is not a $2$-sphere or a torus. If S is incompressible and non-separating, then $S$ is isotopic to an embedded least area minimal surface. 
\end{cor}
 This corollary provides us with enough regularity to obtain estimates similar to the ones for stable incompressible minimal surfaces in the closed hyperbolic case. In [\citenum{hwClosedMinCusped}, Theorem 1.1], Huang and Wang show that for each cusped $M$, there is an explicit height $\tau$ to truncate cusps off $M$ to obtain a compact manifold $M_{\tau}$ with boundary so that any least area closed incompressible non-separating surface stays in $\mt$ [\citenum{hwClosedMinCusped}, Theorem 5.9, Corollary 5.7] (see \Cref{subsecTruncateMtau}). Readers familiar with \cite{bdNorms} may propose the following generalization of \eqref{bdi} for a cusped manifold $M$: 
\[
 	\ds \frac{\pi}{\sqrt{\text{vol}(M)}} \| \cdot \|_{Th} \leq \| \cdot \|_{L^2} \leq \frac{c}{\sqrt{\inj(M_{\tau})}} \|\cdot \|_{Th} \text{ on } \lh(M),
\]
where $c$ is constant. However, we have abandoned this direction in view of the following observation. \Cref{surmtau} combined with arguments from \tn{[\citenum{bdNorms}, Section 4]} suffices to give a similar upper bound by $\max\{\frac{c}{\sqrt{\inj(M_{\tau})}}, \frac{c}{\inj(M_{\tau})}\}$ in the cusped case. However, the constant $c$ is cumbersome, and the boundary $\partial M_{\tau}$ gives rise to terms of a different order. This simplistic approach fails to take advantage of the thick-thin decomposition and the warped product structure of the cusp neighborhoods, and does not provide the more intrinsic upper bound by systole as in \Cref{mainCor}. 

Using properties of subharmonic functions and modified Bessel functions, we prove a version of the theorem independent of the truncation:
\begin{theorem}\label{maintheo}
For all orientable cusped $3$-manifolds $M$, we have 
\begin{equation}\label{mainequ}
	\ds \frac{\pi}{\sqrt{\vol(M)}} \| \cdot \|_{Th} \leq \| \cdot \|_{L^2} \leq  \max \Bigl\{ \frac{10\pi\sqrt{2}}{\sqrt{\sys(M)}}, 4.86\pi\sqrt{1+\frac{d^2}{2}} \Bigr\} \|\cdot \|_{Th} \text{ on } \lh,
\end{equation}
where $d$ is the largest intrinsic diameter of the boundary tori of the maximal cusp neighborhoods and $\sys(M)$ denotes the length of the shortest closed geodesic in $M$.
\end{theorem}
\begin{rem}
	Note that when $M$ is closed, its systole equals twice the injectivity radius. Thus the coefficient $\frac{10\pi\sqrt{2}}{\sqrt{\sys(M)}}$ is just $\frac{10\pi}{\sqrt{\inj(M)}}$ from \eqref{bdi}. Let $\{M_t\}_{t \geq \tau}$ be a compact exhaustion for $M$. When $\psi \in H^1(M;  \Z)$ is harmonic and has infinite $L^2$-norm, we can still bound the growth rate of $\|\psi\|_{L^2(M_t)}$ using its topological complexity, which counts the intersection of $\partial M$ and $S \in H_2(M, \partial M)$ dual to $\psi$. See \Cref{growthRateInfiniteL2Intersect}.
\end{rem}
Using Dehn fillings and Thurston-J{\o}rgensen theory, a natural corollary is
\begin{cor}\label{mainCor}
		Let $V_0$ be a fixed positive constant.	For all but finitely many cusped manifolds $M$ with volume less than $V_0$, we have 
	\begin{equation}\label{maincor}
		\| \cdot \|_{L^2} \leq \frac{10\pi\sqrt{2}}{\sqrt{\sys(M)}} \|\cdot \|_{Th} \text{ on } \lh(M).
	\end{equation}
\end{cor}

\subsection{Topology and the Thurston norm of harmonic forms}
In Section \ref{harformCompCohomo}, we analyze the topology of the space of $L^2$ harmonic $1$-forms on a cusped $3$-manifold $M$, using the vector space isomorphism $\lh \cong \fma$ from \cite{zucker1982, mpHodgeHyp}. We prove that this isomorphism is actually an isometry: 
\begin{prop}
	Let $M$ be a cusped $3$-manifold. The space of $L^2$ harmonic $1$-forms $(\lh(M), \|\cdot\|_{L^2})$ is isometric to $(\fma, \|\cdot\|_{L^2})$. 
\end{prop}
In \Cref{topologyL^2}, we characterize the surfaces dual to $\lh$ as the incompressible and non-peripheral closed surfaces in $M$, which are the analogue of the surfaces representing the second homology of a closed, hyperbolic manifold. The Thurston norm on $H^1(M) \cong H_2(M, \partial M)$ is a genuine norm. A natural corollary (\Cref{genuinenormlem}) on the Thurston norm restricted to $\lh$ is the following: 
\begin{lem}
Let $M$ be a cusped $3$-manifold. Then the Thurston norm $\| \cdot \|_{Th}$ on $\mathcal{H}^1 \cong \text{Im}(H_2(M) \rightarrow H_2(M, \partial M))$ is a genuine norm. 
\end{lem}
In general, the Thurston norm is only a semi-norm, which makes it impossible to bound the $L^2$-norm from above by the Thurston norm. That the underlying metric is hyperbolic is important for \eqref{bdi} and \Cref{maintheo}. For example, on the flat $3$-torus obtained from gluing of the unit cube $\{0 \leq x,y,z \leq 1\}$ in $\R^3$, the $L^2$-norm of $[dz]$ is $1$ while its Thurston norm is $0$. In elliptic manifolds, there are no incompressible surfaces by considering the fundamental groups, and there are no $L^2$ harmonic $1$-forms by Bochner type arguments or Bonnet-Myers theorem. One interesting question that Chris Leininger has asked is on which of the eight geometries we have inequalities similar to \eqref{bdi}. 

\subsection{Geometry and the least area norm of harmonic forms}
After analyzing the topology of the surfaces dual to $\lh$, we show that their geometry is also regular. In particular, their least area representatives do not go deep inside the cusp. The existence of least area representatives of surfaces dual to $\lh$ has recently been established in \cite{hwClosedMinCusped, chmrMinNoncompactExist, errata-chrSurBundle}. We show how to truncate the manifold $M$ canonically to obtain a compact subdomain $\mt$ in \eqref{mtau}. All the least area surfaces dual to classes in $\lh$ are contained in $\mt$ (\Cref{surmtau}). Their areas are bounded from below and above by their genera (\Cref{areagenusprop}), as in \cite{hass1995acy, chmrMinNoncompactExist}. Thus we can define the least area norm and show that it is controlled by the Thurston norm:
\begin{cor}
Let	$M$ be a cusped $3$-manifold. Then 
	\begin{equation}
		\pi \| \cdot  \|_{Th} \leq \| \cdot \|_{LA} \leq 2\pi \|  \cdot\|_{Th} \text{ on } \lh.
	\end{equation} 
\end{cor}
We then define the $L^1$-norm on $\lh$ and show that it is equal to the least area norm in \Cref{LAL1}. This follows from the general principle that \pc duality is an isometry for H\"older-conjugate norms (\Cref{pcdualityIso}).
\subsection{Sharpness and the connection between harmonic forms, foliations, and flows}
It is a natural question whether  \eqref{bdi} and  (\ref{mainequ}) can be realized. Brock and Dunfield show that the left-hand inequality of  \eqref{bdi} is qualitatively sharp, by constructing a sequence of manifolds $M_n$ and $\phi_n \in H^1(M_n)$ such that $\frac{\| \phi_n \|_{Th}}{\| \phi_n \|_{L^2}\sqrt{\vol(M_n)}}$ is a constant. The manifolds $M_n$ cover a fixed manifold $M$. The Thurston norm scales linearly in terms of the degree of the cover by the deep work of Gabai \cite{gdFolitopo} while the $L^2$-norm scales by the square root of the degree.  Quantitatively, Anil and Dunfield \cite{anilnathan} find examples of $M$ on which the left-hand inequality of \eqref{bdi} is $99.6\%$ sharp. Thus it is surprising that we prove that the left-hand of \eqref{bdi} is never realized. This holds similarly in the cusped case. 
\begin{theorem}
Let $M$ be a closed or cusped $3$-manifold. Then 
\begin{center}
	$\ds \frac{\pi}{\sqrt{\vol(M)}} \| \alpha \|_{Th} < \| \alpha \|_{L^2}$,
\end{center} 
where $\alpha \in H^1(M)$ for closed $M$ and $\alpha \in \lh(M)$ for cusped $M$. 
\end{theorem}
 When $M$ is closed, one obstruction for the equality to be realized is that a harmonic $1$-form on a hyperbolic manifold cannot have constant length. In particular, this follows from the fact that the second fundamental form of a minimal surface in a space form is the real part of a holomorphic quadratic differential, with $4g-4$ zeros, and cannot be everywhere nonzero. In a private communication, F. Pallete pointed out an obstruction coming from the area of a stable minimal surface in hyperbolic $3$-manifolds, which constitutes another point of view. For details, see \Cref{noFormConstLength}. \par  
 We provide two additional points of view in \Cref{sharpFunctionalSection}. One is from dynamic theory \cite{noflowgeodesichyper}, which proves the non-existence of smooth vector fields whose flow lines are all geodesic. The other concerns the non-existence of geometric foliations \cite{wwNonexigeometricflow}. The noncompact case is easier due to the exponential decay of a harmonic $1$-form near infinity (\Cref{decayharmonicform}). We see a nice interaction between harmonic forms, minimal surfaces, and foliations in Section \ref{sharpFunctionalSection}. We end by defining two functionals on $\Psi\coloneqq$ \{Closed or cusped orientable hyperbolic $3$-manifolds\} in \Cref{functional}:
\begin{center}
	$\ds D_i(M)\coloneqq \inf_{\alpha \in \lh}\frac{\pi \|\alpha\|_{Th}}{\sqrt{\textnormal{vol}(M)}\|\alpha \|_{L^2}}$,
\end{center}
and 
\begin{center}
	$\ds	D_s(M)\coloneqq \sup_{\alpha \in \lh}\frac{\pi \|\alpha\|_{Th}}{\sqrt{\textnormal{vol}(M)}\|\alpha \|_{L^2}}$.
\end{center}
These functionals provide a unifying point of view for several results regarding sharpness and the examples constructed in \cite{bdNorms}. They also enable us to ask questions that do not come up in the original inequalities and may be interesting invariants to study. 

\section{\texorpdfstring{$L^2$}{L2} Harmonic Forms and Compactly Supported Cohomology}\label{harformCompCohomo}

\subsection{Basic definitions}
Let $M$ be a complete smooth Riemannian manifold of dimension $n$. We denote the space of smooth differential $k$-forms on $M$ by $\Omega^k(M)$ and the subspace of forms with compact support by $\Omega^k_0(M)$. When the underlying manifold $M$ is clear, we denote the two spaces by $\ok$ and $\okc$. The exterior derivative is denoted by $d$. 

\begin{defi}
	The $k$-th de Rham's cohomology group of $M$ is defined by: 
	\begin{center}
	$\ds H^k_{dR}(M)=\frac{\{\phi \in \Omega^k(M): d \phi=0\} }{d\Omega^{k-1}(M)}$.
 \end{center}
\end{defi}
By de Rham's theorem, we may identify $H^k_{dR}(M)$ with $H^k(M)$, the singular cohomology, where we always use real coefficients unless specified otherwise. 

\begin{defi}
	The $k$-th compactly supported cohomology of $M$ is defined by: 
	\begin{center}
		$\ds H^k_{0}(M)=\frac{\{\phi \in \Omega^k_0(M): d \phi=0\} }{d\Omega^{k-1}_0(M)}$.
	\end{center}
\end{defi}
Next, we define $L^2$ forms. Denote by $dx_1$, $dx_2$, ..., $dx_n$ a local orthonormal basis for the $1$-forms. A $k$-form $\phi$ can be expressed locally as 
\begin{center}
	$\phi = \sum\limits_{I=\{i_1<i_2<...<i_k\}} \phi_I dx_{i_1}\wedge dx_{i_2}\wedge ... \wedge dx_{i_k}$,
\end{center}
where $\phi_I \in L^2_{loc}$ are locally square-integrable functions. The space $L^2\Omega^k(M)$ consists of $\phi \in \Omega^k$ such that
\begin{center}
	$\ds \|\phi \|^2_{L^2} \coloneqq  \int_M \phi\wedge *\phi =\int_M |\phi(x)|^2 d \vol_g(x)<\infty$.
\end{center}
 The space $L^2 \Omega^k(M)$ is a Hilbert space with inner product 
\begin{center}
	$\ds \langle \phi, \psi \rangle  =\int_M \phi\wedge *\psi =\int_M (\phi(x), \psi(x))_g d\vol_g(x)$.
\end{center}
The formal adjoint of $d$ on $\okc$ is:
\begin{center}
	$d^*:\okhc \rightarrow \okc$
\end{center}
which is defined by the following condition: 
\begin{center}
$\forall \phi \in \okhc$ and $\psi \in \okc,$ \par
$\langle d^*\phi, \psi \rangle = \langle \phi, d\psi \rangle$.
\end{center}
Note that $d^*$ can also be defined on $k$-forms using the Hodge star operator by 
\begin{center}
$d^* \phi =(-1)^{n(k-1)+1} *d* \phi$.
\end{center}
The Hodge Laplace operator 
\begin{equation}
	\Delta =dd^*+d^*d: \okc \rightarrow \okc
\end{equation}
is a second-order symmetric elliptic operator. When the metric is complete, by \cite{gaffney1955hilbert}, the $L^2$ closure of the operator $\Delta$: 
\begin{equation}
	\Delta: L^2 \ok \rightarrow L^2 \ok
\end{equation}
is essentially self-adjoint. Its domain is $\{\omega \in L^2\ok: d\omega, d^*\omega, dd^*\omega, d^*d\omega \in L^2 \}$.
\begin{defi}\label{harmonicClosedCoclosed}
	The $L^2$ harmonic $k$-forms are defined by 
\begin{center}
		$ \mathcal{H}^k(M)=\{\phi \in L^2\Omega^{k}(M): \Delta \phi=0 \}$,
\end{center}
where $\Delta=d d^*+d^*d$. The above definition is equivalent to 
\begin{center}
	$ \mathcal{H}^k(M)=\{\phi \in L^2\Omega^{k}(M): d\phi= d^*\phi=0 \}$,
\end{center}
when the metric on $M$ is complete, by \tn{\cite{gaffney1955hilbert}}. 
\end{defi}
Using the Hodge-de Rham decomposition of $L^2\Omega^k(M)$ (also called the Kodaira decomposition) [\citenum{differentiable1984MR760450}],
\begin{center}
	$L^2\Omega^k(M)=\mathcal{H}^k(M) \oplus \overline{d\Omega^{k-1}_0(M)} \oplus \overline{d^*\Omega^{k+1}_0(M)}$.
\end{center}

\begin{defi}\label{rabc}
	Relative and absolute boundary conditions for forms: Consider a manifold $M$ with boundary $\partial M$, with $\iota: \partial M \rightarrow M$ the natural inclusion map. For a $k$-form $\alpha$, the relative boundary condition is $\iota^* \alpha =0$, while the absolute boundary condition is $\iota_{\nu} \alpha =0$, where $\iota_{\nu}$ is the interior multiplication with the inward-pointing unit normal vector field $\nu$ on $\partial M$. 
\end{defi}
The relative boundary condition resembles Dirichlet boundary condition $f|_{\partial M}=0$ for a function $f$, while the absolute boundary condition resembles Neumann boundary condition $\frac{\partial f}{\partial \nu }=0$, where $\nu$ is the inward-pointing unit normal vector field on $\partial M$. \par 
When $M$ is the interior of a compact manifold $\overline{M}$ with compact boundary $\partial \overline{M}$, i.e., $M=\overline{M}-\partial \overline{M}$, the $k$th cohomology with compact support $H^k_0(M)$ is isomorphic to the relative cohomology group of $\overline{M}$ [\citenum{mpHodgeHyp}, (3.1), (3.2)]: 
\begin{equation}\label{relativecoho}
	H^k_0(M)=H^k(\overline{M}, \partial \overline{M}) \coloneqq \frac{\{\phi \in \ok(\overline{M}): d\phi=0, \iota^*\phi=0\}}{\{d\psi : \psi \in \okl(\overline{M}), d\psi=0, \iota^*\psi=0\}}
\end{equation}
where $\iota: \partial \overline{M}\rightarrow \overline{M}$ is the inclusion map. 

\subsection{Hodge theory on cusped \texorpdfstring{$3$}{3}-manifolds}\label{hodgeL^2}\par
In the remainder of this section, $M$ denotes a closed or cusped hyperbolic $3$-manifold unless noted otherwise. When $M$ is cusped, topologically, it is homeomorphic to the interior of an irreducible compact $3$-manifold $\overline{M}$ with toroidal boundary. The (topological) boundary of $M$ refers to $\partial \overline{M}$. This coincides with the compactification of cusped manifolds presented in \cite{mpHodgeHyp}. This subsection covers Hodge theory on cusped $3$-manifolds, in particular for the space of $L^2$ harmonic $1$-forms $\lh$. Most of the ideas are adapted from \cite{mpHodgeHyp}, which builds a Hodge theory for  geometrically finite hyperbolic $n$-manifolds with possibly infinite volume, generalizing the results on cusped manifolds from \cite{zucker1982}, and obtains asymptotic estimates for $L^2$ harmonic $k$-forms near infinity. The particular case that we are interested in is the following \cite{mpHodgeHyp, zucker1982}:  
\begin{lem}\label{fmaIma}
 Let $M$ be a cusped $3$-manifold. Then $\lh$ is isomorphic to $\text{Im}(H^1_0(M) \rightarrow H^1(M))$ as vector spaces.
\end{lem}
Note that in general, there is no relationship between harmonic forms and compactly supported cohomology or its image in $H^1$. For example, consider the hyperbolic disc $D=\mathbb{H}^2$. Its space of harmonic $1$-forms $\lh(D)$ is of infinite dimension, while $H^1(D)$ and $H^1_0(D)$ are trivial. As another example, the product space $S^1 \times \R$ has trivial $\lh$ but nontrivial $H^1_0$ and $H^1$. These two examples are taken from a nice survey [\citenum{jdHarmonicFormComp}, p. 292], which also discusses other subtle issues related to harmonic forms. \par 

From \cite{mpHodgeHyp}, the isomorphism in \Cref{fmaIma} can be interpreted in the following way. Take $\alpha \in \lh$; by doing a perturbation supported near infinity, we obtain $\phi \in H^1_0$ such that $\alpha$ is cohomologous to $\phi$ in $H^1$. The specific construction is in \Cref{l2harIsometricCompact} where we promote the isomorphism above to an isometry. From basic algebraic topology, $H^1_0(M)$ is isomorphic to $H_2(M)$ [\citenum{AH}, Theorem $3.35$]. There is a subtle but important difference between $\text{Im}(H^1_0(M) \rightarrow H^1(M))$ and $H^1_0(M)$, which is crucial for the main theorem. Consider a boundary-parallel torus $T$ in a cusp. It is clear that $T$ defines a class $[T] \in H_2(M)$, and a $1$-form $\alpha$ dual to $[T]$ can be chosen to be compactly supported around $T$. If $M$ has more than one cusps (boundary components), then $[T]$ is nontrivial in $H_2(M)$. The $1$-form $\alpha$ has Thurston norm zero, and it is not obvious that its harmonic norm is zero. Moreover, as $T$ exits to infinity, its area approaches zero. Fortunately, $[T]=0$ in $H_2(M, \partial M)$ regardless of the number of boundary components of $M$, and hence  $[T]=0$ in $\text{Im}(H_2(M) \rightarrow H_2(M, \partial M))$, which is equivalent to $[\alpha]=0$ in $\text{Im}(H^1_0(M) \rightarrow H^1(M))$. We can also see this from geometry by computing the $L^2$-norm. Lift $T$ to the upper half-space model $\Hy=\{(x, y, z)|z>0\}$ so that a component of its cover is $z=c$. Let $\alpha=f(z)/z \, dz$ be a $1$-form dual to $T$ where $f(z)$ is a bump function supported in $[c-1, c+1]$ with $\int f/z \, dz=1$. The $L^2$-norm of $\alpha$ is
\begin{center}
	$\ds \| \alpha \|^2_{L^2} = \int_M \alpha \wedge * \alpha = \int_{c-1}^{c+1} f^2 \frac{dx\wedge dy\wedge dz}{z^3} = \mathcal{O} \left(\log \frac{c+1}{c-1} \right) \rightarrow 0$ as $c \rightarrow \infty$.
\end{center}
The third equality above is a simple estimate following from $\int f/z \, dz=1$. Thus if $[\alpha] \neq 0$ in $\mathcal{H}^1(M)$, one of the components of the closed surface dual to $\alpha$ must have genus at least $2$, as there are no essential spheres or essential tori. \par 
We now use the isomorphism in \Cref{fmaIma} to obtain a Fourier series expansion of $\alpha$ inside a cusp neighborhood $\C$. By composing with an isometry if necessary, we can assume that a component $\widetilde{\C}$ of the cover of $\C$ in $\Hy$ is given by $\tilde{\C}=\{(x, y, z)|z \geq \sqrt{2}\}$. The boundary torus $\T = \partial \C$ is lifted to $z=\sqrt{2}$ and inherits an intrinsic flat metric. The fact that $\alpha \in \lh$ is cohomologous in $H^1$ to the image of the inclusion map of $\phi \in H^1_0$ implies $\alpha = df$ on $\C$. Moreover, we take advantage of the warped product structure of the cusp. Let $K_j(\cdot)$ denote the modified Bessel functions \cite{asMathFunctionsFormula} and $\psi_i$ the $i$-th eigenfunction of $\Delta$ on the torus $\T$ with nonzero eigenvalue $-\lambda_i^2$ ($\lambda_0=0$). \par  
\begin{lem}\label{L2harmonic=df}
	Let $\alpha$ be an $L^2$ harmonic $1$-form. Then restricted to a cusp neighborhood $\C$, there exists a harmonic function $f$ such that 
	\begin{equation}\label{harmonicDf}
		\alpha = df,
	\end{equation}
 whose Fourier expansion is:
	\begin{equation}\label{phidfBessel}
		f(x,y,z)=\sum_{i=1}c_i zK_1(\lambda_i z)\psi_i(x,y)
	\end{equation}
for some $c_i \in \R$. 
\end{lem}
\begin{proof}
	Since $\alpha \in \lh$ is cohomologous to $\phi \in \fma$, we have 
	\begin{equation}
		\alpha - \phi =df \tn{ for some function } f.
	\end{equation}
	Using \Cref{topologyL2harmonic} whose proof does not depend on this lemma, $\phi$ is dual to a surface whose support is disjoint from the cusp neighborhood $\C$, so that $\phi=0$ on $\C$. This implies that $\alpha  =df$ on $\C$. 
	Since $\alpha$ is harmonic, a simple computation shows that $f$ is also harmonic on $\C$, using the completeness of the metric of $M$ and the equivalence in \Cref{harmonicClosedCoclosed}. To obtain the Fourier expansion, we use the fact that $\T$ is flat, oriented, and $2$-dimensional. In particular, the Hodge decomposition of $\T$ implies $\Omega^1(\T)= d\Omega^0 \oplus (d^* *\Omega^0) \oplus\lh(\T)$. Thus the vector space of eigen $1$-forms on $\T$ has a basis consisting of $\{d\psi_{i}, \: *d \psi_{i}\}$. 
	Let $\alpha=\gamma(x,y,z,dx,dy)+\beta(x,y,z)dz=df(x,y,z)$.  Using [\citenum{mpHodgeHyp}, (4.7)], 
	\begin{equation}
		\ds \alpha = \sum_{i=0}^{\infty} \gamma_i(z) d\psi_i(x,y)+  \sum_{i=0}^{\infty} \beta_i(z) \psi_i(x,y) dz.
	\end{equation}
	Combined with the identity $\frac{d}{dz}zK_1(z)= -zK_0(z)$ ([\citenum{asMathFunctionsFormula}, 9.6.28]) and [\citenum{mpHodgeHyp}, (4.10)], we have \eqref{phidfBessel}. 
\end{proof}
Although the function $f$ is harmonic inside each cusp neighborhood and possesses similar series expansions, it is not globally harmonic by the following argument. Since it decays to $0$ in each cusp neighborhood, using the thick-thin decomposition, $f$ is bounded. A bounded harmonic function on a complete finite-volume manifold is constant by \cite{yau1976some,yau1976someErrata}. By the properties of modified Bessel functions, both the $1$-form $\gamma$ and the function $f$ decay exponentially. This is [\citenum{mpHodgeHyp}, Theorem 4.12], which we reformulate as follows:  
\begin{theorem}\label{decayharmonicform}
If $\alpha \in \lh$, then on a neighborhood of a rank-$2$ cusp, $\alpha = \gamma+\beta dz$ where $\gamma$ is a $1$-form and $\beta$ is a function, both depending parametrically on $z$. Asymptotically, they satisfy 
\begin{center}
$|\gamma|=\mathcal{O}(e^{-\lambda z})$ \tn{ and } $\beta=\mathcal{O}(e^{-\lambda z})$,
\end{center}
for some $\lambda>0$. 
\end{theorem}
 See also \cite{jdHarmonicFormComp}, which discusses the exponential decay of the $L^2$ forms in warped product manifolds. For readers familiar with number theory, we note that this is consistent with, for example, [\citenum{elstrodt2013groups}, Theorem 3.1]. \par 
 The isomorphism $\lh \cong \text{Im}(H^1_0(M) \rightarrow H^1(M))$ is actually an isometry. Equip $\fma$ with a natural $L^2$-norm by the following definition: 
\begin{equation}
	\|\phi\|_{L^2} \coloneqq \inf\{\|\psi\|_{L^2} \tn{ for } \psi \in H^1_0, \textnormal{ cohomologous to } \phi \textnormal{ in } H^1\}.
\end{equation}
\begin{prop}\label{l2harIsometricCompact}
	Let $M$ be a cusped $3$-manifold. The space of $L^2$ harmonic $1$-forms $(\lh(M), \|\cdot\|_{L^2})$ is isometric to $(\fma, \|\cdot\|_{L^2})$. 
\end{prop}
\begin{proof}
	The proof follows from the proof of [\citenum{mpHodgeHyp}, Theorem 3.13], where they prove that the two vector spaces are isomorphic. Fix a cusp neighborhood $N = [t, \infty) \times \mathbb{T}^2$, with coordinate $(z, x)$. The first step is to compactify the cusp at infinity and introduce $s=1/z$ as a new coordinate. Then $N=[0, \epsilon] \times \mathbb{T}^2$ for some small $\epsilon >0$. Now $M$ is treated as a compact manifold with boundary, with $s=0$ corresponding to $\partial N$. Near $\partial N$, a harmonic $1$-form $\alpha$ can be written as 
	\begin{equation}
		\alpha = \gamma(s)+h(s)ds,
	\end{equation}
	where $h \in C^{\infty}(N)$. Here $\gamma, h$ are viewed as forms and functions depending parametrically on $s$. By [\citenum{mpHodgeHyp}, (3.3)], we define the retraction operator, 
	\begin{equation}\label{retrationOperator}
		\ds R\alpha(s)=\int_0^s h(t)dt,
	\end{equation} 
  so that near $s=0$, 
\begin{equation}
	\alpha=dR\alpha +\gamma(0)
\end{equation}
 for $\alpha$ closed. Take a cutoff function $f(s)$ equal to $1$ near $s=0$, supported in $N$. Extending by $0$,  $f(s)R\alpha$ is globally defined. With $i:\partial N \rightarrow N$ the inclusion,  $i^*(f(s)R\alpha)=0$ and $\alpha-d(fR\alpha)$ vanishes near $\partial N$. We apply the above construction in each of the cusp components. The map 
\begin{equation}
F: \lh \rightarrow \fma, F(\alpha)=\alpha-d(hR\alpha)
\end{equation}
is an isomorphism as proven in [\citenum{mpHodgeHyp}, Theorem 3.13]. We will use functions $f_i$ with increasingly small support to construct a sequence $\alpha_i$  of smooth forms with compact support such that 
\begin{equation}
	\|\alpha -\alpha_i \|_{L^2(M)} \rightarrow 0.
\end{equation}
By computation, near $s=0$
\begin{center}
$\ds \|\alpha\|_{L^2(N)}^2=\int_0^{\epsilon} |\gamma(s)|^2\frac{1}{s} +|h(s)|^2 s^3 ds$.
\end{center}
Since $\alpha \in L^2$, we have $\gamma(s)=\gamma(0)=0$ near $s=0$. 
Define $N_i \coloneqq [0, \frac{\epsilon}{i}) \times \mathbb{T}^2$, and let $f_i$ is a cutoff function equal to $1$ near $s=0$ with $\tn{supp}(f_i)\subset N_i$ and $ |df_i|\leq \frac{i}{\epsilon}$. Define 
\[
	\alpha_i =\alpha - d(f_i R\alpha).
\]
Then we have
\begin{align*}
	\|\alpha -\alpha_i \|_{L^2} = \|d(f_i R\alpha)\|_{L^2}^2 &=\int_0^{\epsilon/i}s^3 \left|\partial_s f_i(s)\right|^2 \left| \int_0^s h(t)dt \right| ^2+ |f_i(s)|^2 |h(s)|^2 ds\\
	&\leq \int_0^{\epsilon/i}s^3 \cdot \frac{i^2}{\epsilon^2} \left| \int_0^s h(t)dt \right| ^2+|h(s)|^2 ds\\
	&\leq \frac{\epsilon^3}{i^3}\frac{i^2}{\epsilon^2}	C\rightarrow 0\\,
\end{align*}
as $i \rightarrow \infty$, where we use $h \sim \mathcal{O}(e^{-\lambda/s})$ from \Cref{decayharmonicform}. 
\end{proof}

\subsection{Topology of \texorpdfstring{$L^2$}{L2} harmonic \texorpdfstring{$1$}{1}-forms and the Thurston norm}\label{topologyL^2}
 By \eqref{relativecoho}, $H^1_0(M) \cong H^1(\overline{M}, \partial \overline{M}) \cong H_2(M)$. It is a classical fact that every nontrivial $[S] \in H_2(M)$ for both closed and cusped $M$ is represented by a (possibly disconnected) closed smoothly embedded incompressible surface. This follows from \pc duality and the regular value theorem \cite{normThurston}. We will henceforward assume every class in $H_2$ is represented by such a closed smoothly embedded incompressible surface. 
\begin{defi}
	A class $[S] \in H_2(M)$ is peripheral if it is homologous to a union of boundary components. Otherwise, it is called non-peripheral. 
\end{defi}
 We use the absolute value of the Euler characteristic to measure the topological complexity of classes whose representatives are surfaces with genus $\geq 2$. For a connected surface $S$, $\chi_-(S) \coloneqq \max\{-\chi(S),0\}$, where $\chi$ is the Euler characteristic. Extend this to disconnected surfaces via $\chi_-(S \sqcup S') =\chi_-(S) + \chi_-(S')$. For a compact irreducible $3$-manifold $M$, the Thurston norm \cite{normThurston} of $\alpha \in H^1(M; \mathbb{Z}) \cong H_2(M, \partial M; \mathbb{Z})$ is defined by 
\begin{center}
	$\| \alpha \|_{Th}= \min\{\chi_- (S) \suchthat S\text{ is a properly embedded surface dual to } \alpha \}$.
\end{center}
The Thurston norm extends to $H^1(M; \mathbb{Q})$ by making it linear on rays through the origin, and then extends by continuity to $H^1(M; \R)$. When $M$ is closed and hyperbolic, it is non-degenerate and a genuine norm. It is only a semi-norm on $H^1_0(M) \cong H_2(M)$ for $M$ with two or more cusps. However, the story is different for $\lh \cong \text{Im}(H_2(M) \rightarrow H_2(M, \partial M))$. We first give a classification of $S \in \Ima$:
\begin{enumerate}
	\item $S$ Peripheral. There exists a proper map $u: (M, \partial M) \rightarrow (I, \partial I)$, where $I$ is a closed interval. If we denote by $[dx]$ a generator for $H^1_0(I)$, then the $1$-form $\alpha$ dual to $S$ satisfies $[\alpha] = [u^*(dx)]$. Moreover, $S$ is homologous in $H_2$ to some boundary-parallel tori and hence represents the trivial class in $H_2(M, \partial M)$. Thus a nontrivial $\alpha \in \mathcal{H}^1$ cannot be dual to a closed peripheral surface. 
		
\begin{figure}[ht]
\centering
\begin{subfigure}{.5\textwidth}
\centering
\includegraphics[width=5cm]{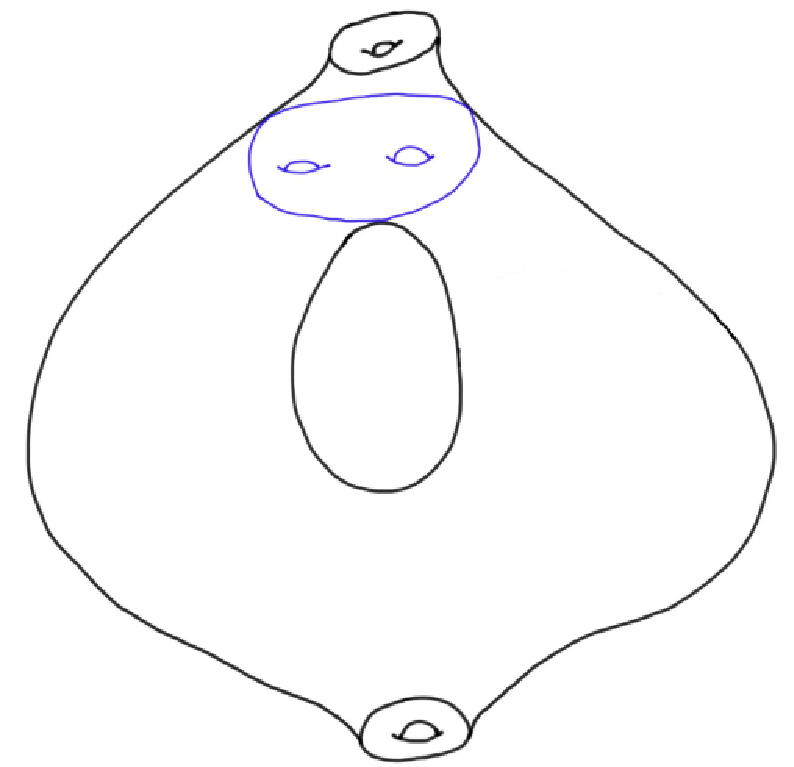}
\caption{Peripheral, trivial and $\|\cdot \|_{L^2}=0$}			
\end{subfigure}%
\begin{subfigure}{.5\textwidth}
\centering
\includegraphics[width=5cm]{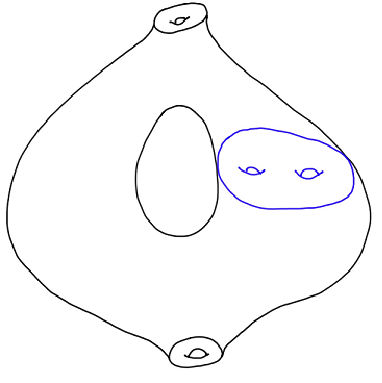}
\caption{Non-peripheral, nontrivial and $\|\cdot \|_{L^2} \neq 0$}
\label{fig:sub2}
\end{subfigure}
\caption{Classification of $\Ima$}
\end{figure}

	\item \label{nonperipheral} $S$ Non-peripheral. By $H_2(M, \partial M) \cong [M: S^1]$, $S$ is the fiber of a regular value of a smooth map $u: M \rightarrow S^1 = \R /\Z$. Take $u$ to be harmonic or the energy-minimizing representative in its homotopy class. The nontrivial harmonic $1$-form $\alpha = u^*(d\theta)$ is dual to $S$, where $[d \theta]$ is a standard generator for $H^1(S^1)$.
\end{enumerate}
We have obtained 
\begin{lem}\label{topologyL2harmonic}
	A nonzero $L^2$ harmonic $1$-form is dual to a non-peripheral surface.
\end{lem}
 Conversely, for a noncompact incompressible surface $S$ with finite topology whose ends are properly embedded, like a thrice punctured sphere embedded inside a cusped $3$-manifolds \cite{adams1985thrice}, the $L^2$-norm of its dual $1$-form $\alpha$ must be infinite. This is well-known to the experts, but we provide a short proof for completeness. 
\begin{lem}\label{L2normblowupsurlem}
	If $[S] \in H_2(M, \partial M) \backslash \Ima$, then a harmonic $1$-form $\alpha$ dual to $S$ does not have finite $L^2$-norm. 
\end{lem} 
\begin{proof}
If the $1$-form $\alpha$ dual to $S$ has finite $L^2$-norm, by \Cref{l2harIsometricCompact} and $\fma \cong \Ima$, $S$ must be homologous to the image of a class $[S'] \in H_2(M)$, a contradiction. 
\end{proof}
In fact, if we let $\{M_t\}$ denote a certain compact exhaustion of $M$, then we can derive a lower bound on $\| \alpha \|_{L^2(M_t)}$ from the number of algebraic intersections of $[S]$ and $[\partial M]$. Let $S$ be a surface dual to $\alpha$ whose number of intersections with $\partial M$ realizes $\# [S] \cap [\partial M]$. Let $\Hy= \{(x,y,z)|z>0\}$.  Assume that an end of $S$ in a truncated cusp $\C$ is lifted to a disk asymptotic to $z=\infty$ by composing with isometries of $\Hy$ if necessary. For simplicity, assume the base of $\C$ is a unit square. Parameterize $\C$ by $\{( x, y, z )\in \mathbb{H}^3 \suchthat (x,y)\in [0,1]\times [0,1], z \geq c \}$. Let $\C_t$ denote points in $\C$ whose $z$-coordinate is $\leq t$. The ends of $S$ are homotopic to totally geodesic, infinite annuli inside the cusps (see, for example, \cite{chrSurBundle}). Assume for now $S$ only has one end in $\C$.  After a change of coordinate, we also suppose  that this end is parameterized by $\{ (1/2, y, z)\in \C \}$. The several assumptions made above do not change the order of growth. Then the (asymptotic) lower bound on $\|\alpha\|_{L^2(\C_t)}^2$ is $\bo(\log t)$.  On each horizontal torus $T_z$, $\| \alpha \|^2$ integrates to approximately $z^2$ along a Euclidean geodesic $\{ (x, 1/2, z): x\in [0,1]\}$ intersecting the end exactly once. Integrating in the $z$-direction, we have
	\begin{center}
		$ \displaystyle \| \alpha \|_{L^2(\C_t)}^2  \gtrsim \int_c^{t} \int_0^1 \int_0^1 z^2 \frac{1}{z^3} dxdydz \sim \log t$.
	\end{center}
 If $S$ has $k$ algebraic intersections with $\partial \C$, since $S$ is properly embedded, the $k$ ends must be parallel, in which case $\| \alpha \|_{L^2(\C_t)}^2  \gtrsim k \log t$. If $M$ has only one cusp $\C$, then we define $M_t$ to be the union of $\C_t$ with the thick part. If $M$ has several cusps, we define similarly $M_t$ to be the union of the thick part with an exhaustion of all the cusps relative to the boundary of an (embedded) maximal cusp neighborhood, where $t$ corresponds to $z$-coordinate in $\Hy$. Thus, we have proven the following. 
\begin{prop}\label{growthRateInfiniteL2Intersect}
	Let ${M_t}$ be the above compact exhaustion of $M$. If $\alpha \in H^1(M; \Z)$ is harmonic and dual to $[S] \in H_2(M, \partial M) \backslash \Ima$, then 
	\[
	\displaystyle \| \alpha \|_{L^2(M_t)} \gtrsim \bo(\sqrt{\log t}).
	\]
\end{prop}
This is consistent with [\citenum{mpHodgeHyp}, (4.9)], whose proof is based on harmonic analysis. 
A criterion for the Thurston norm to be a genuine norm is the following \cite{normThurston}: For $H_2(M; \Z)$ (or $H_2(M,\partial M; \Z)$), if every embedded surface representing a nonzero element has negative Euler characteristic, then $\| \cdot \|_{Th}$ is a norm. When $M$ is cusped, all incompressible tori are boundary-parallel, and there is no incompressible surface with positive Euler characteristic. Thus the following holds.
\begin{prop}\label{genuinenormlem}
	When $M$ is a cusped $3$-manifold, $\| \cdot \|_{Th}$ on $\mathcal{H}^1 \cong \text{Im}(H_2(M) \rightarrow H_2(M, \partial M))$ is a genuine norm. 
\end{prop}
In addition to the Thurston norm, we need the $L^1$- and $L^{\infty}$-norms.  For a general $L^p$ theory, see \cite{scott1995Lpform}. The $L^{\infty}$-norm of an $L^2$ harmonic $1$-form $\alpha$ is defined by
\begin{equation}
	\ds \|\alpha\|_{L^{\infty}}=\max_{p\in M}|\alpha_p|.
\end{equation}
 Since $\alpha$ decays exponentially near infinity and $M$ is finite-volume, its $L^{\infty}$-norm is well-defined.  We also define the following $L^1$-norm on $\fma$:
\begin{equation}\label{L1norm}
	\|\psi\|_{L^1} =\inf \Bigl\{\int_M |\phi| \tn{ for } \phi\in \Omega^1_0(M) \tn{ representing } \psi \in \fma \Bigr\}.
\end{equation}
When $M$ is closed, the definition of the $L^1$-norm is just: 
\begin{equation}
	\|\psi\|_{L^1} =\inf \Bigl\{\int_M |\phi| \tn{ for } \phi\in \Omega^1(M) \tn{ representing } \psi \in H^1(M)\Bigr\}.
\end{equation}
The $L^1$-norm behaves very differently from the $L^2$-norm on the cohomology space, as it tends to be realized by non-smooth forms. A sequence of forms $\phi_i$ with compact support converging to $\phi$ realizing the $L^1$-norm of $[\phi]$ is similar to a sequence of bumps functions converging to a Dirac functional. Later we will use the Cauchy-Schwarz inequality to bound the $L^1$-norm of $\alpha \in \lh$ by its $L^2$-norm. To bound the $L^2$-norm of $\alpha$, we convert it into some integral of $\alpha$ on a surface $S$ dual to $\alpha$. When the manifold $M$ is closed, it is clear by \pc duality that integrating against a closed integral $1$-form is equal to integrating against its dual surface. Specifically, fix a surface $S$ dual to $\phi\in H^1(M; \Z)$ and let $\alpha$ be the harmonic representative of $\phi$. Then we have $\int_M \beta \wedge \alpha =\int_S \beta$ for every closed $2$-form $\beta$. When $M$ is closed, if $\alpha$ is harmonic, it is closed and coclosed. When $M$ is noncompact, a similar conclusion on $\alpha$ holds if the metric is complete. Let $\alpha \in \lh$ correspond to $\phi \in H^1_0(M, \Z)$, dual to $[S] \in \Ima$. By \pc duality for compactly supported cohomology [\citenum{btDifferentialAlgTopo}, I.5], we have again integral equality $\int_M \beta \wedge \phi =\int_S \beta$ for every closed $2$-form $\beta$. From this, we deduce 
\begin{prop}[\pc duality for cusped manifolds]
	Let $M$ be a cusped $3$-manifold. Let $\alpha \in \lh$ correspond to $\phi \in H^1_0(M, \Z)$, dual to a surface $S$. Then we have 
	\begin{equation}\label{homologicaldualitysur}
		\ds \|\alpha\|^2_{L^2}= \int_S *\alpha.
	\end{equation}
\end{prop} 
\begin{proof}
 Since $\phi$ has compact support, $\int_M  *\alpha \wedge \phi =\int_S *\alpha$.  It then suffices to show $\int_M *\alpha \wedge (\alpha -\phi) =0$. Since $[\phi-\alpha]=0$ in $H^1$, $\phi-\alpha = df$ for $f \in C^{\infty}(M)$. If we can show
 \begin{equation}\label{f*alphaFiniteInt}
 	  \int_M f *\alpha <\infty,
 \end{equation}
 then
	\begin{center}
		$\ds \int_M *\alpha \wedge (\alpha -\phi)=\int_M d(*\alpha \wedge f)=0$,
	\end{center}
	by [\citenum{yau1976some}, Lemma] which generalizes Stokes' theorem to noncompact manifolds. Apply the Cauchy-Schwarz inequality to \eqref{f*alphaFiniteInt}. Since $\alpha \in L^2$, the proof is finished if $f \in L^2$. As $\phi-\alpha = df$ and $\phi$ has compact support,  $df =\alpha$ outside a large ball $B$. Thus 
	\begin{equation}
		\ds \|f\|^2_{L^2(M)} \sim \|f\|^2_{L^2(M-B)} \sim \int_c^{\infty}\frac{f^2}{z^3}dz.
	\end{equation}
	Since $M$ has finite volume and $df =\alpha$ decays exponentially by \Cref{decayharmonicform}, $\|f\|_{L^2(M)}$ is finite.
\end{proof}

\section{Minimal Surfaces and Least Area Norm}
In this section, we discuss minimal surfaces and their properties in closed and cusped $3$-manifolds. A closed $2$-sided surface is minimal if its mean curvature is $0$. Let $\nu$ denote a unit normal vector field on a surface $S$ with a trivial normal bundle. A normal vector field $X$ may be written as $f\nu$, where $f$ is some smooth function. The stability operator $L$ is defined as 
\begin{equation}\label{stop}
	Lf=\Delta_S f+|\sigma|^2 f+Ric_M(\nu,\nu)f,
\end{equation}
where $Ric_M$ is the Ricci tensor of the ambient manifold $M$, $\sigma$ is the second fundamental form on $S$, and $\Delta_S$ is the intrinsic Laplacian operator of $S$. A minimal surface $S$ is stable if for all compactly supported variations $F$ with  fixed boundary, 
\begin{center}
	$\ds \left. \frac{d^2}{dt^2}\right|_{t=0}Area(F(S,t))=-\int_S\langle F_t,LF_t\rangle dA\geq 0$.
\end{center}
This is equivalent to the stability operator being negative semidefinite. A least area surface in a homotopy class is necessarily stable, but not vice versa. The area of a least area surface $F$ with genus $g \geq 2$ in a closed or cusped hyperbolic manifold $M$ satisfies the following inequalities ([\citenum{hass1995acy}, Lemma 6], [\citenum{mrMinHypLA}, Section 8.1] and [\citenum{chmrMinNoncompactExist}, Remark 3]):
\begin{equation}\label{surface/area/genus}
	2\pi (g-1) \leq area(F) \leq 4\pi (g-1);
\end{equation}
this was originally observed in an unpublished work of Uhlenbeck \cite{ukleastAreahyp}. The left-hand inequality in \eqref{surface/area/genus} comes from the stability of $F$ in a hyperbolic $3$-manifold, while the right-hand inequality follows directly from the Gauss equation for a minimal surface in a space form. The existence of least area surface representatives in the homology class of $[S] \in H_2(M)$ for $M$ compact is well-understood (see Federer \cite{gmtMR0257325}). Freedman, Hass, and Scott \cite{freedman1983least} show that in a compact $3$-manifold, a least area surface in the homotopy class of an embedded surface is embedded. The existence of least area incompressible closed surfaces in the cusped setting requires techniques other than the usual minimizing techniques in compact settings in \cite{syExistIncomMinTopo, minImmerUhlenSack}, and was recently settled by Huang and Wang \cite{hwClosedMinCusped} and by Collin, Hauswirth, Mazet, and Rosenberg \cite{chmrMinNoncompactExist, errata-chrSurBundle}. A cusped $M$ is geometrically noncompact and the surfaces can enter the cusps arbitrarily deeply. A cusp bears some similarity with a short tube in that they are both foliated by intrinsic Euclidean tori. This foliation is useful in some area estimates used in cut-and-paste arguments, which provide obstructions for the least area closed surface to enter arbitrarily deeply inside the cusp. If the surface is sufficiently deep, the part inside the cusp will give rise to a large area contribution. One then cuts off the cusp at a torus at some height, obtains a surface with some boundary components, and glues the disks coming from the cutoff torus. This reduces the area. Similar ideas are also used in \cite{hwComlengShortCurve} to investigate minimal foliation questions on hyperbolic $3$-manifolds fibering over $S^1$. In conclusion, we have
\begin{theorem}\label{lacm}\textnormal{[\citenum{hwClosedMinCusped}, Corollary 1.2]}
Let $S$ be a closed orientable embedded surface in a cusped $3$-manifold $M$ without torus components. If $S$ is incompressible and non-separating, then $S$ is isotopic to an embedded least area minimal surface.
\end{theorem} 
\subsection{Truncation of \texorpdfstring{$M$}{M} and the definition of \texorpdfstring{$\mt$}{Mt}}\label{subsecTruncateMtau}
A surface $S \subset M $ dual to $\alpha$ is called \textit{taut} if $S$ is incompressible, realizes $\|\alpha \|_{Th}$, and no union of components of $S$ is separating. We can always take the taut representative in a homology class. Using $\lh(M) \cong \Ima$, we will show there is a compact manifold $M_{\tau}$ with flat toroidal boundary, containing all the least area surfaces dual to classes in $\lh(M)$. We now describe the construction of $M_{\tau}$, following \cite{hwClosedMinCusped}. Suppose $M$ has $k$ cusps. A cusp neighborhood is maximal if there are no larger embedded cusp neighborhoods containing it, which occurs exactly when it is tangent to itself at one or more points. We assume cusp neighborhoods are always cut out by geometric horotori. Parameterize the maximal cusp neighborhoods by $\C_i =\T_i \times [0, \infty),  i=1,\dots, k$ (here $\tau \in [0,\infty)$ measures the intrinsic geometric distance from $\T_i \times \{0\})$. Let $\tau_0 >0$ be the smallest number such that each cusped region $\T_i \times [\tau_0, \infty)$ is disjoint from any other maximal cusp region of $M$. For $\tau \geq \tau_0$, let $\mt$ be the compact subdomain of $M$ defined by: 
\begin{equation}\label{mtau}
	\ds \mt=M- \cup_{i=1}^k (\T_i \times (\tau, \infty)).
\end{equation}
By construction, $\mt$ is a compact submanifold of $M$ with boundary components which are concave with respect to inward-pointing normal vectors. We lift $M$ to $\Hy$ with the upper half space model to extract quantitative information about the boundary tori, essentially the translational distances corresponding to the associated parabolic isometries. For each $i$ with $1\leq i \leq k$, we lift $M$ in a way so that the horoball $H_i$ corresponding to $\C_i$ is centered at $\infty$, and $\partial H_i =\{(x,y,z)\in \Hy \suchthat z=1\}$. Suppose $\Gamma_i$ is the rank-$2$ parabolic subgroup corresponding to $\C_i$ generated by $p \mapsto p+ \xi_i$ and $p \mapsto p+\eta_i$, where $\xi_i$ and $\eta_i$ are non-trivial $\mathbb{R}$-linearly independent complex numbers. For each $i$, we choose ($\xi_i$, $\eta_i$) to be a pair of generators which is minimal with respect to the lexicographic ordering of moduli of pairs of complex numbers. Now define 
\begin{equation}\label{bdconst}
	L_0=\max\{e^{\tau_0}, |\xi_1|+|\eta_1|, \dots, |\xi_k|+|\eta_k|\} >0 \tn{ and } \tau=\log(3L_0).
\end{equation}
 Then $M_{\tau}$ contains all the least area surfaces described in \Cref{lacm} and $\tau$ is independent of these surfaces ([\citenum{hwClosedMinCusped}, Remark 2.3]). We thus have

\begin{prop}\label{surmtau}
	Let $\alpha \in \lh$ be nontrivial. There is a least area representative $S \in \Ima$ dual to $\alpha$ such that $S \subset M_{\tau}$, where $M_{\tau}$ is a compact $3$-manifold with toroidal boundary truncated from $M$ and $\tau$ is given by \eqref{bdconst}. 
\end{prop}
Since a least area surface $S$ is contained in a compact thick part of the noncompact $M$, it is not surprising that we can bound its area by its genus analogous to the compact setting. 
\begin{prop}\label{areagenusprop}
Let $S$ be a closed, orientable, incompressible, and non-peripheral embedded surface of genus $\geq 2$ in a cusped $3$-manifold. If $S$ is least area in its homotopy class, then
\begin{equation}
		2\pi (g-1) \leq area(S) \leq 4\pi (g-1).
\end{equation}
\end{prop}
The proof is in [\citenum{chmrMinNoncompactExist}, Remark 3]. 

\subsection{The least area norm and the \texorpdfstring{$L^1$}{L1}-norm}
We are now ready to define the least area norm. We first define it with integral coefficients and then extend it by linearity and continuity. Denote $\{\alpha \in \lh \suchthat \alpha \tn{ corresponds to } \phi \in H^1_0(M; \Z)\}$ by $\lh(M; \Z)$. 
\begin{defi}
For $\alpha \in \lh(M; \Z)$, let $\mathcal{F}_{\alpha}$ be the collection of smooth maps $f: S \rightarrow M$ where $S$ is a closed oriented surface with $f_*([S])$ dual to $\alpha$. The least area norm of $\alpha$ is 
\begin{center}
$\|\alpha \|_{LA}=\inf\{\text{Area}(f(S)) \suchthat f\in \mathcal{F}_{\alpha}\}$.
\end{center} 
\end{defi}
With this definition, a corollary of \Cref{areagenusprop} is
\begin{cor}\label{lathcor}
	Let $M$ be a cusped $3$-manifold. Then: 
	\begin{equation}
		\pi \| \cdot \|_{Th} \leq \|\cdot \|_{LA} \leq 2\pi \| \cdot \|_{Th} \text{ on } \lh(M; \Z).
	\end{equation}
\end{cor}

To bound $\| \cdot \|_{LA}$ by $\| \cdot \|_{L^2}$, we need the following lemma. We remind the readers that, in general, the $L^1$-norm on the cohomology class of a harmonic form $\alpha$ is not equal to $\int_M |\alpha|$.
\begin{lem}\label{LAL1} For $\phi \in H^1_0(M;\Z)$, the least area norm and the $L^1$-norm satisfy
	\begin{center}
			$\|\phi \|_{LA}=\|\phi \|_{L^1}$.
	\end{center}
\end{lem}
\begin{proof}
For the full proof when $M$ is closed, see [\citenum{bdNorms}, Lemma 3.1]. We mention some necessary modifications for the noncompact case. For $\phi \in H^1_0(M;\Z)$, a dual surface $S$ is also smoothly embedded and can be chosen to be contained in $\mt$. We can construct a $1$-form compactly supported near $S$ whose $L^1$-norm is very close to $\|\phi\|_{LA}$ and give an upper bound on $\|\phi\|_{L^1}$. For $\phi$ dual to boundary-parallel tori exiting to $\infty$, it is clear that $\| \phi \|_{LA}=0=\| \phi \|_{L^1}$. It is slightly more complicated to prove $\|\phi \|_{LA} \leq \|\phi \|_{L^1}$. We use the discussion on non-peripheral surfaces before \Cref{topologyL2harmonic} to get a smooth map $u:M \rightarrow S^1$ so that $\phi=u^*(dt)$. The remaining arguments are identical. 
\end{proof}

Using the above Lemma, the least area norm extends continuously from $Im(H^1_0(M; \Z) \rightarrow H^1(M;\Z))$ to a semi-norm on all of $\fma \cong \lh$. Let $\alpha \in \lh$. Take a sequence $\phi_i \in H^1_0$ of $1$-forms with compact support cohomologous to $\phi$, as in the proof of \Cref{l2harIsometricCompact}. Now 
\[
	\|\alpha \|_{LA} =\|\phi_i \|_{LA}=\|\phi_i \|_{L^1}\leq \int_M |\phi_i|  \leq \int_M |\alpha |_{L^1} + \epsilon,
\]
where $\epsilon \rightarrow 0$ as $i \rightarrow \infty$. Thus we have 
\begin{equation}\label{laleqL1}
	\|\alpha \|_{LA} \leq \int_M |\alpha|  \tn{ for } \alpha \in \lh. 
\end{equation}
The above Lemma in the closed case [\citenum{bdNorms}, Lemma 3.1] is a special case of a more general principle. The \pc duality for closed manifolds is an isometry when we put $L^p$ and $L^q$ norms on the dual vector spaces where $p$ and $q$ are H\"older conjugates. This is [\citenum{loewnersystolic}, Remark 7.2 (1)], which we provide for convenience: 
\begin{lem}\label{pcdualityIso}
	Let $M$ be a closed $n$-dimensional Riemannian manifold. The \pc duality map
\end{lem}
\begin{center}
	$PD: (H^1(M, \R), \|\cdot\|_{L^p}) \rightarrow (H_{n-1}(M, \R), \|\cdot\|_{L^q})$
	is an isometry. 
\end{center}
Combining with \Cref{LAL1}, the above lemma shows that the least area norm on $H_2$ is the $L^{\infty}$-norm. 

\section{The \texorpdfstring{$L^{\infty}$}{L infinity}-norm, Main Theorem and the Proof}
\subsection{Bounding the \texorpdfstring{$L^{\infty}$}{L infinity}-norm by the \texorpdfstring{$L^2$}{L2}-norm}
To prove the right-hand  side of  (\ref{mainequ}), we need several lemmas.
\begin{lem}[\citenum{bdNorms}, Lemma 4.3]\label{dfv(r)}
	If $f: \mathbb{H}^3 \rightarrow \mathbb{R}$ is harmonic and $B$ is a ball of radius $r$ centered about $p$ then 
	\begin{equation}
		|df_p|\leq \frac{1}{\sqrt{v(r)}} \| df \|_{L^2(B)},
	\end{equation}
where \par
\begin{equation}\label{nurBD}
	v(r)=6\pi (r+2r \textnormal{csch}^2 (r)-\coth (r)(r^2 \textnormal{csch}^2(r)+1) ).
\end{equation}
\end{lem} 
We need the notion of a \textit{Margulis constant}. There are multiple definitions and we use the one from [\citenum{culler2012margulis}, 1.0.2]. Let $\Gamma$ be a discrete and torsion-free subgroup of $\text{Isom}_+(\Hy)$. For $\gamma \in \Gamma$ and $P\in \Hy$, define $d_P(\gamma)=\text{dist}(P, \gamma \cdot P)$. 
\begin{defi}
	A Margulis constant $\mu$ is a number such that if $x, y\in \Gamma$ satisfy
	\[
	\max (d_P(x), d_P(y))< \mu \tn{ for any } P\in \Hy,
	\]
	then $x$ and $y$ commute.
\end{defi}
 A Margulis constant gives rise to a thick-thin decomposition: $M =M_{\text{thin}} \cup M_{\text{thick}}$ where 
\begin{center}
	$M_{\text{thin}} = \{m \in M \suchthat \inj_m < \mu/2 \}$ and $M_{\text{thick}} = \{m \in M \suchthat \inj_m \geq \mu/2 \}$.
\end{center}
We call $M_{\text{thin}}$ the thin part; it is a disjoint union of Margulis tubes and rank-$2$ cusps. When $H^1(M) \neq 0$, $\mu =0.29$ is a Margulis constant by \cite{culler2012margulis}. \par 
 We now generalize [\citenum{bdNorms}, Theorem 4.1] to maximal cusp neighborhoods. Take one such neighborhood $\mathcal{C'}$  and let its boundary torus be $\T' \coloneqq \{z=1\}$. By slightly shrinking the cusp neighborhood $\C'$, we obtain a smaller cusp neighborhood $\C$ whose boundary $\T$ is now based at $z=\sqrt{2}$.

\begin{prop}\label{subharmonicCuspMaxBoundary}
	Let $M$ be a cusped $3$-manifold, with a cusp neighborhood $\C$ and corresponding boundary $\T$ as described above. Denote the intrinsic diameter of $\T$ by $d$. Then we have for $\phi \in \lh$,
	\begin{equation}
		\|\phi\|_{L^{\infty}(\C)} \leq 2.43\sqrt{1+\frac{d^2}{2}}\|\phi \|_{L^2(M)}.
	\end{equation}
\end{prop}
\begin{proof}
Recall that when acting on functions, the sign convention for the Laplace-de Rham operator $\Delta_d$ and the Laplace-Beltrami operator $\Delta_b$ is opposite: $\Delta_b = -\Delta_d$. For the proof of this lemma only, we reserve $\Delta$ for the Laplace-Beltrami operator $\Delta_b$. Using the Bochner technique [\citenum{jjRiemannianGeoAnalysis}, Theorem 4.5.1], we have 
	\begin{equation}\label{bochForm}
		-\Delta_d \langle \phi, \phi \rangle = 2| \nabla \phi |^2+2 \tn{Ric}(\phi, \phi).
	\end{equation}
If $M$ had non-negative Ricci curvature, we would be able to conclude that 
\begin{center}
	$\Delta \langle \phi, \phi \rangle =-\Delta_d \langle \phi, \phi \rangle \geq 0$,
\end{center}
so that $\|\phi\|_{\infty}\coloneqq \|\phi\|_{L^{\infty}(\C)}$ must be achieved on $\partial \C =\T$, by the properties of non-negative subharmonic functions vanishing at infinity and a limit argument. From \Cref{L2harmonic=df}, there exists a harmonic function $f$ on $\C$ such that $\phi=df$. In view of the negative curvature, we use the following identity for a harmonic function $f$:
\begin{equation}\label{rickTrick}
	\Delta f^2=2|\nabla f|^2+2f \Delta f=2|\nabla f|^2.
\end{equation}
 There exists a collection of eigenfunctions $\psi_i$ on $\T$ which form an orthonormal basis for $L^2(\T)$. Each non-constant eigenfunction on $\T$ is orthogonal to the constant functions. Since the Fourier expansions of $f$ in \eqref{phidfBessel} only have nonconstant eigenfunctions, we have
\begin{equation}
	 \int_{\T} f=0.
\end{equation}
Denote the restriction of $f$ to $\T$ by $f_1$. It follows that the range of $f_1$ contains $0$ as an interior point: $f_1$ can not be everywhere non-positive or non-negative. This is required in order to bound $f^2$ by $|df|_{max}$ which we need later. In general, $|df|_{max}$ can only bound $(f_{max}-f_{min})^2$; it also bounds $f^2$ when $f_{max}=-f_{min}$. Thus we pick a constant $c$ and replace $f$ by
\begin{equation}
	 g\coloneqq f+c
\end{equation}
so that $g_1\coloneqq g|_{\T} = (f+c)|_{\T}$ satisfies 
\begin{equation}
	(g_1)_{\max}=-(g_1)_{\min}.
\end{equation}
This constant $c$ also minimizes $\max_{\T} g^2 =\max g_1^2$ among all such choices. Since $0$ is an interior point of the range of $f_1$, $c$ is less than $(g_1)_{\max}$. Now we can bound $\|\phi\|_{\infty}$ by $\|\phi|_{\T}\|_{\infty}$ and $g|_{\T}$. Rewriting \eqref{bochForm} using \eqref{harmonicDf}, we have 
\begin{equation}\label{boch-2}
	\frac{1}{2}	\Delta  |\phi|^2 = | \nabla^2 f |^2-2 | \nabla f|^2.
\end{equation}
Combining with \eqref{rickTrick}, we have 
\begin{equation}
	\Delta (\frac{1}{2} |\phi|^2+(f+c)^2)=|\nabla^2 f|^2\geq 0
\end{equation}
and thus 
$h\coloneqq \frac{1}{2} |\phi|^2+(f+c)^2$ is a subharmonic function. The maximum of a subharmonic function on a compact region is achieved on the boundary. We use a limit argument to show that the maximum of $h$ in $\C$ is achieved on $\T$. As $z=t \rightarrow \infty$, $f$ and $|df|$ approach $0$. Thus $h(z)\rightarrow c^2$. Consider $\C_{\sqrt{2}}^t$, the portion of the cusp neighborhood from $z=\sqrt{2}$ to $z=t$, with two boundary components: $\T$ and $\T_t$. The maximum of $h$ is achieved either on $\T$ or on $\T_t$. Since the maximum of $h|_{\T}=\frac{1}{2} |\phi_{\T}|^2 +g_1^2$ is larger than $c^2$, by choosing a large enough $t$, we can arrange that 
\[
	\ds \max_{\T_t} h < 	\max_{\T} h.
\]
It then follows that 
\begin{equation}
		\ds \max_{\C} h = 	\max_{\T} h.
\end{equation}
Now we have 
\begin{align*}
	\|\phi\|_{\infty}^2 &\leq 2 \max_{\C} (\frac{1}{2} |\phi|^2+g^2)\\
	&=2 \max_{\T} (\frac{1}{2} |\phi|^2+g^2)\\
	&\leq \max_{\T} |\phi|^2 + 2\max_{\T} g^2\\.
\end{align*}
The $c$ we choose implies 
\begin{equation}
	\max_{\T}g^2 =(\frac{g_{\tn{max}}-g_{\tn{min}}}{2})^2,
\end{equation}
where $g_{\tn{max}}$ and $g_{\tn{min}}$ are the maximum and the minimum of $g$ restricted to $\T$, respectively. 
Using the mean value theorem on $\T$, we have
\begin{equation}
	g_{\tn{max}}-g_{\tn{min}} \leq d\cdot |dg|_{\tn{max}}.
\end{equation}
Here we also use the fact that the torus $\T$ is a submanifold and $|dg|_{\T} \leq |dg|_{hyp}$ pointwisely. 
Thus we have
\begin{equation}
		\|\phi\|_{\infty}^2 \leq (1+\frac{d^2}{2})\max_{\T} |\phi|^2.
\end{equation}
Using the lower bound on waist size from \cite{acWaistSizeI}, we have $\tn{inj}_{\T}$ is larger than $0.48$. The largest ball embedded in $\C'$ is centered at $\T$ and has hyperbolic radius $\ln \sqrt{2}$. By \Cref{dfv(r)}, we have 
\begin{equation}
	\ds\max_{\T} |\phi| \leq \frac{1}{\sqrt{\nu(\ln(\sqrt{2}))}} \|\phi\|_{L^2(M)} \approx 2.43\|\phi\|_{L^2(M)},
\end{equation}
where $\nu(r)$ is defined in \eqref{nurBD}. 
\end{proof}
Now we have 
\begin{prop}
For a cusped $3$-manifold $M$, we have 
\begin{equation}\label{LinftyL2Cusped}
		\|\cdot \|_{L^{\infty}} \leq \max\{ \frac{5\sqrt{2}}{\sqrt{\sys(M)}}, 2.43\sqrt{1+\frac{d^2}{2}} \} \|\cdot \|_{L^2} \text{ on } \lh,
\end{equation}
where $d$ is the largest diameter of the boundary tori of the maximal cusp neighborhoods. 
\end{prop}
\begin{proof}
Assume $\lh \neq 0$. For $\mu= 0.292$, each cusp component of $M_{thin}$ is contained in a maximal cusp to which \Cref{subharmonicCuspMaxBoundary} applies. For $\phi \in \lh$, if $\|\phi\|_{L^{\infty}}$ is not realized in any of the cusp neighborhoods, then it is realized in the interior of the compact part $M_C$ which is the complement of the union of the maximum cusp neighborhoods. Now applying \Cref{dfv(r)} and [\citenum{bdNorms}, Theorem 4.1], we have $\|\phi \|_{L^{\infty}} \leq \frac{5}{\sqrt{\inj(M_C)}}\|\phi \|_{L^2} = \frac{5\sqrt{2}}{\sqrt{\sys(M)}}\|\phi \|_{L^2}$.
\end{proof}

\subsection{Main theorem and the proof}
We are ready to prove \\
\textbf{\Cref{maintheo}} \textit{For all cusped hyperbolic $3$-manifolds $M$ one has} 
\begin{center}
		$\ds \frac{\pi}{\sqrt{\vol(M)}} \| \cdot \|_{Th} \leq \| \cdot \|_{L^2} \leq  \max\{ \frac{10\pi\sqrt{2}}{\sqrt{\sys(M)}}, 4.86\pi\sqrt{1+\frac{d^2}{2}} \} \|\cdot \|_{Th} \text{ on } \lh,$
\end{center}
\textit{where $d$ is the largest diameter of the boundary tori of the maximal cusp neighborhoods.}\\
Here we present a proof whose argument is similar to  \cite{bdNorms} using local estimates and properties of minimal surfaces in hyperbolic $3$-manifolds. 
\begin{proof}
	Let $\alpha \in \lh$. The left-hand inequality follows from an application of the Cauchy-Schwarz inequality. By \Cref{lathcor} and \eqref{laleqL1} we have 
	\begin{center}
		$ \ds \pi \| \alpha \|_{Th} \leq \| \alpha \|_{LA} \leq \int_M |\alpha| \leq \| \alpha \|_{L^2} \| 1 \|_{L^2}= \| \alpha \|_{L^2} \sqrt{\vol(M)}$.
	\end{center}
	Thus we have 
	\begin{center}
		$\ds \frac{\pi}{\sqrt{\vol(M)}} \| \alpha \|_{Th} \leq \| \alpha \|_{L^2}.$
	\end{center}
Now we prove the right-hand inequality. By linearity and continuity of the two norms, it suffices to prove it for $\alpha \in \lh(M; \Z)$. Fix a closed, taut, and least area surface $S$ dual to $\alpha$. Using \eqref{homologicaldualitysur}, we have
\begin{align*}
	\|\alpha\|^2_{L^2} & =\int_S *\alpha \leq \int_S |*\alpha| dA\\
	& =\int_S |\alpha| dA \leq \int_S \|\alpha\|_{L^\infty} dA  \\ 
	& = \|\alpha\|_{L^\infty} Area(S) \leq 2\pi \|\alpha \|_{L^\infty} \|\alpha \|_{Th},
\end{align*}
where we use \Cref{lathcor} in the last line. Using \eqref{LinftyL2Cusped} and dividing both sides of the above inequality by $\|\alpha\|_{L^2}$ completes finish the proof.
\end{proof}

\Cref{mainCor} follows naturally: \\
\textbf{\Cref{mainCor}} \textit{	Let $V_0$ be a fixed positive constant.	For all but finitely many cusped manifolds $M$ with volume less than $V_0$, we have}
	\begin{equation}
		\| \cdot \|_{L^2} \leq \frac{10\pi\sqrt{2}}{\sqrt{\sys(M)}} \|\cdot \|_{Th}.
	\end{equation}
\begin{proof}
	As the waist size is at least $1$, if the diameter $d$ of a cusp tends to infinity, then the volume of the corresponding cusp goes to infinity. This implies that the diameter $d$ is bounded in terms of $V_0$. On the other hand, by the Thurston-J{\o}rgensen theory, the number of hyperbolic $3$-manifolds with volume $\leq V_0$ and $ \tn{systole} \geq \epsilon_0$ is finite, where $\epsilon_0 >0$ is a constant. Thus for all but finitely many cusped manifolds $M$ with volume $\leq V_0$, the systole term in the right-hand side of  \eqref{mainequ} dominates the diameter term. 
\end{proof}

\section{Sharpness of the Inequalities and a Functional Point of View}\label{sharpFunctionalSection}
In this section, we assume the manifold is closed as in \eqref{bdi} unless stated otherwise. We first study the left-hand inequality of  \eqref{bdi}.
\begin{theorem}\label{noHarmonicSharpConstant}
	There does not exist a closed hyperbolic $3$-manifold $M$ such that 
	\begin{center}
		$\ds \frac{\pi}{\sqrt{\vol(M)}} \| \alpha \|_{Th} = \| \alpha \|_{L^2}$ for some $\alpha \in H^1$.
	\end{center}
\end{theorem}
We start by investigating various conditions that the $1$-form $\alpha$ and its dual surface $S$ must satisfy for the equality to hold. 
\subsection{Three conditions for the sharpness and their interactions}
Recall that if $\phi \in H^1$ corresponds to $\alpha \in \lh$, the $L^1$-norm on the cohomology class of $\phi$ is $\|\phi\|_{L^1}=\inf_{\phi'\sim \phi \in H^1} \int_M |\phi'|$. Using \Cref{harformconstlength}, let $S$ be a closed minimal surface whose tangent bundle $TS \subset \ker(\alpha)$.
\begin{prop}\label{sharpnessleftprop}
	If $\alpha \in \lh$ is a harmonic $1$-form such that 
	\begin{center}
		$\ds \frac{\pi}{\sqrt{\vol(M)}} \| \alpha \|_{Th} = \| \alpha \|_{L^2}$,
	\end{center}
	it must satisfy the following three conditions: 
	\begin{enumerate}
		\item \label{c1} It has constant length: $| \alpha_p |= c$ on $M$; 
		\item \label{c2} Its $L^1$-integral coincides with the $L^1$-norm of its cohomology class:	
		\begin{equation}
		\int_M |\alpha| = \| \alpha \|_{L^1} ;
		\end{equation}
		\item \label{c3} It must satisfy 
		\begin{equation}
		Area(S)=\pi\chi_-(S).
		\end{equation}
		 This is equivalent to $\int_S | \sigma |^2 = 2Area(S)$ or $\int_S h_{11}^2+h_{12}^2= \int_S 1$, where $\sigma$ is the second fundamental form associated to $S$ with components $h_{11}$, $h_{12}$, $h_{21}$, $h_{22}$. 
	\end{enumerate}
\end{prop}
The proof of the above proposition is primarily bookkeeping; see the proof of [\citenum{bdNorms}, Theorem 1.2, 3.2]. Condition \ref{c1} is more likely to hold on Seifert fibered manifolds than hyperbolic ones; an example is the $1$-form $dz$ on the $3$-torus $\mathbb{T}^3$ obtained from the gluing of a standard unit cube in the Euclidean space. Condition \ref{c2} also seems unlikely, as commented in \cite{bdNorms}: $\|\phi\|_{L^1}$ is typically realized by a non-smooth form, supported like a Dirac measure on the surface. Condition \ref{c3} is not possible as discussed in \Cref{noFormConstLength}. Our first observation is that Condition \ref{c1} implies Conditions \ref{c2} and \ref{c3}. 
The next lemma is a straightforward corollary of [\citenum{sysgeometrytopo}, Proposition 16.9.1]. 
\begin{lem}
Let $M$ be a closed Riemannian $n$-manifold. Let $\alpha$ be a nontrivial harmonic $1$-form with constant length. Then it realizes the $L^1$-norm of its cohomology class. 
\end{lem}
\begin{proof}
From [\citenum{sysgeometrytopo}, Proposition 16.9.1], we have that existence of a nontrivial harmonic $1$-form $\alpha$ with constant length implies that 
\begin{equation}
\|\alpha \|_{L^1} \frac{1}{\textnormal{vol}(M)}= \|\alpha\|_{L^2} \frac{1}{\sqrt{\textnormal{vol}(M)}},
\end{equation}
which implies
\[
\|\alpha \|_{L^1} = \sqrt{\vol(M)}\|\alpha\|_{L^2}.
\]
If $\alpha$ has constant length $c$,  
\begin{center}
	$\ds \|\alpha \|_{L^1} = \vol(M)c =\int_M |\alpha| $.
\end{center}
\end{proof}
 Now we prove that Condition \ref{c1} implies Condition \ref{c3} on a closed hyperbolic $3$-manifold. The proof is partially inspired by the proof of [\citenum{sternScalarHarmonic}, Theorem 1.1].
\begin{lem}\label{areaSEulercharacteristic}
	Let $M$ be a closed hyperbolic $3$-manifold. If $\alpha$ is a nontrivial $L^2$ harmonic $1$-form with constant length, then there exists a closed minimal surface $S$ dual to $\alpha$, such that
	\begin{center}
		Area$(S)= \pi \chi_-(S)$.
	\end{center}
\end{lem}
\begin{proof}
	Let $|\alpha_p|=c$, where $c$ is a constant. By the Bochner formula for $1$-forms, we have 
	\begin{center}
		$\frac{1}{2} \Delta |\alpha|^2 =| \nabla \alpha |^2 -2|\alpha|^2 \implies | \nabla \alpha |^2=2|\alpha|^2=2c^2$.
	\end{center}
	As a closed form, $\alpha=df$ locally in a ball. Denote $\nabla f$ by $V$. Using \Cref{harformconstlength}, there is a closed minimal surface $S$ whose tangent bundle $TS \subset \ker(\alpha)$. Since $|V|=|\alpha|$ is constant, we have for all vector fields $W$ on $M$, 
	\begin{equation}\label{Hessian}
		0=\langle \nabla_W V, V\rangle=\textnormal{Hess}(f)(W, V).
	\end{equation}
	The Hessian of $f$, $\textnormal{Hess}(f)=\{h_{ij}\}$, is a symmetric $3 \times 3$ matrix, with the first two columns corresponding to vectors tangent to $S$, and last column to vectors parallel to $V$. \Cref{Hessian} implies that the last column and the last row consist of zeros. Thus we have $| \nabla \alpha |=| \nabla \alpha|_S |$.
	The second fundamental form of $S$ is 
	\begin{equation}\label{gra1form2nd}
		\ds \sigma_S =\frac{\nabla \alpha}{|\alpha|} \Big\vert_S,
	\end{equation}
	which implies that 
	\begin{center}
		$\ds |\sigma_S|^2=2 \implies \int_S h_{11}^2+h_{12}^2 = \frac{1}{2}\int_S |\sigma_S|^2=\int_S 1$.
		
	\end{center}
	The rest of the arguments are identical to the proof of [\citenum{hass1995acy}, Lemma 6]. Using the Gauss-Bonnet theorem, the Gauss formula and the hyperbolicity of the metric, we have 
	\begin{center}
		$\int_S h_{11}^2+h_{12}^2 = -2\pi\chi(S)+\int_S -1=-Area(S)-2\pi \chi(S)$,
	\end{center}
	which gives $Area(S)= -\pi \chi(S)=\pi \chi_-(S)$. 
\end{proof}

\subsection{Proof of non-sharpness using harmonic \texorpdfstring{$1$}{1}-forms and minimal foliations}
Now we verify our intuition about constant length harmonic $1$-forms: such forms do not exist in hyperbolic manifolds. Manifolds all of whose harmonic forms are constant length are interesting in their own right, and have connections with systolic geometry. See \cite{lenharmonicform} and citations therein. We provide four different points of view, each producing a proof of the following proposition.
\begin{prop}\label{noFormConstLength}
	Let $M$ be a closed hyperbolic $3$-manifold. No nontrivial harmonic $1$-forms of $M$ have constant length. 
\end{prop}
\begin{proof}
Using [\citenum{mrMinHypLA}, 8.1] and [\citenum{gjSurSpaceForm}, Theorem 12], a (connected) immersed stable orientable closed minimal surface $S$ of genus $g$ in a closed hyperbolic $3$-manifold cannot have area $2\pi(g-1)=\pi \chi_-(S)$, contradicting \Cref{areaSEulercharacteristic}. If $S$ dual to $\alpha$ has several components, note that $\chi_-(\cdot)$ is additive. This provides the first proof. \par 
Alternatively, the zeros of a holomorphic quadratic differential provide an obstruction to the existence of harmonic $1$-forms of constant length. By the proof of \Cref{areaSEulercharacteristic}, the existence of such a $1$-form implies that the second fundamental form $\sigma_S$ satisfies $|\sigma_S|^2=2$ on the dual surface $S$. By [\citenum{sysgeometrytopo}, Lemma 16.7.1], $S$ is minimal. It is a classical fact known to Bianchi that the minimality of $S$ and constant sectional curvature of $M$ imply that 
\begin{equation}
	\sigma_S = \textnormal{Re}(h_{11}-ih_{12})(dz)^2=\textnormal{Re} \, \eta,
\end{equation}
where $\eta =(h_{11}-ih_{12})(dz)^2$ is an invariantly defined holomorphic quadratic differential (see [\citenum{closedminSurfaceHyperbolic}, Theorem 4.1] and \cite{lbComMinSurS3}). Any such differential must vanish at exactly $4g-4 \geq 4$ points [\citenum{FMprimerMCG}, p. 312], where $g=g(S)$. Thus it is impossible that $|\sigma_S|^2=2$ identically. This completes the second proof.
\end{proof}

We now describe obstructions from flows and foliations, which provide the last two proofs. We first need a lemma which relates harmonic $1$-forms of constant length to minimal foliations [\citenum{sysgeometrytopo}, Lemma 16.7.1]: 
\begin{lem}\label{harformconstlength}
	If $\alpha$ is a harmonic $1$-form of constant length, then the leaves of the distribution ker($\alpha$) are minimal surfaces, and the vector field $V$ that is dual
	to $\alpha$ has geodesic flow lines orthogonal to ker($\alpha$).
\end{lem}
Combining this with the following result of Zeghib, we obtain an obstruction from dynamical system theory and provide a third proof of \Cref{noFormConstLength}: 
\begin{theorem}
\tn{(Zeghib \cite{noflowgeodesichyper})} There is no smooth vector field on a closed hyperbolic $3$-manifold all of whose flow lines are geodesics. As a corollary, there do not exist harmonic $1$-forms of constant length on $M$. 
\end{theorem}
Note that the flow lines given by harmonic $1$-forms are smooth, while Zeghib only requires continuity. On the other hand, the minimal foliations described in \Cref{harformconstlength} are specific examples of the geometric foliations defined in [\citenum{wwNonexigeometricflow}, Definition 1.4]: 
\begin{defi}\label{geometricfoliation}
	Let $S$ be a closed surface, $\epsilon >0$ a constant, and $h$ an embedding 
		\begin{equation}
		h: (-\epsilon, \epsilon) \times S \rightarrow M.
	\end{equation}
 We say $h$ is a geometric $1$-parameter family of closed minimal surfaces if
	\begin{enumerate}
		\item $h$ is $C^2$ with respect to both $t$ and $p\in S$.
		\item for all $ t$, the leaf $h_t \subset M$ is a minimal surface.
		\item for all $p\in S$, $f(t, p)\coloneqq \langle (h_t)_*(\partial_t), \nu \rangle |_{t=0}$ only depends on the principal curvature of $S$ at p. One writes $f(0,p)=f(0, \|\sigma_S\|^2(p))$ where $\sigma_S$ is the second fundamental form of $\{0\} \times S$ at $(0,p)$ in $M$. 
		\item $f(0, \cdot) : S\rightarrow \R$ is not identically $0$. 
	\end{enumerate}
\end{defi}
The foliation in \Cref{harformconstlength} with orthogonal geodesic flow lines is the strongest geometric foliation one can hope for in \Cref{geometricfoliation} and corresponds to $f \equiv 1$. The non-existence of such foliations is proved in the following. 
\begin{theorem}\tn{[\citenum{wwNonexigeometricflow}, Theorem 1.4]} 
	A closed hyperbolic $3$-manifold does not admit a geometric foliation by minimal surfaces, and hence no harmonic $1$-forms have constant length.
\end{theorem}
This provides the fourth and last proof of \Cref{noFormConstLength}. For the cusped case, to achieve equality, the same arguments go through and show that we still need to satisfy the three conditions in \ref{sharpnessleftprop}. Condition \ref{c2} that $\alpha$ has a constant operator norm is no longer possible since it decays exponentially near the cusp. 

\begin{cor}
	There does not exist a cusped $3$-manifold $M$ such that 
	\begin{center}
		$\ds \frac{\pi}{\sqrt{\vol(M)}} \| \alpha \|_{Th} = \| \alpha \|_{L^2}$ for some $\alpha \in \lh$.
	\end{center}
\end{cor}

\subsection{A functional point of view}\label{functional}
We define two functionals $D_i$ and $D_s$ on the collection $\Psi \coloneqq$ \{Closed or cusped hyperbolic  $3$-manifolds\}. It unifies various phenomena in this paper and in \cite{bdNorms} and allow us to ask new questions. Define 
\begin{equation}
	\ds	D(\alpha) \coloneqq  \frac{\pi \|\alpha\|_{Th}}{\sqrt{\textnormal{vol}(M)}\|\alpha \|_{L^2}} \tn{ for } \alpha \in \lh(M)-\{0\}.
\end{equation}
Then we define
\begin{equation}
\ds	D_i(M)\coloneqq \inf_{\alpha \in \lh} D(\alpha)\tn{ and } \ds	D_s(M)\coloneqq \sup_{\alpha \in \lh}D(\alpha).
\end{equation}
The left-hand inequality of  (\ref{mainequ}) and the rigidity discussed in the last section can be phrased in terms of $ D_s$. Let $B_H$ denote the unit sphere in $(\lh(M; \R), \|\cdot \|_{L^2} )$.
\begin{prop}
	Let $M$ be a closed or cusped hyperbolic $3$-manifold. Then $D_s(M) <1$. 
\end{prop}
\begin{proof}
	Let $\alpha_j \in \lh(M)$ such that $D(\alpha_j) \rightarrow 1$. Define $\widetilde{\alpha_j} \coloneqq \alpha_j/\|\alpha_j\|_{L^2}$ so that $\widetilde{\alpha_j} \in B_H$. Since the Thurston norm is linear, we have $D(\widetilde{\alpha_j}) \rightarrow 1$. Since $\lh(M, R)$ is finite-dimensional for both closed and cusped $M$, $B_H$ is compact. Thus $\widetilde{\alpha_j} \rightarrow \widetilde{\alpha} \in B_H$ such that $D(\widetilde{\alpha})=1$. This contradicts \Cref{noHarmonicSharpConstant}. 
\end{proof}
From the proof, we also deduce that there exist $\alpha_i, \; \alpha_s \in B_H$ such that $D_i(M)=D(\alpha_i)$ and $D_s(M)=D(\alpha_s)$. Using the right-hand side of  \eqref{bdi} or  (\ref{mainequ}), we have 
\begin{center}
	$\ds \frac{\|\cdot\|_{Th}}{\|\cdot\|_{L^2}} \geq \frac{\sqrt{\textnormal{inj}(M)}}{10\pi}$ or  $\ds 1/\max\{ \frac{10\pi\sqrt{2}}{\sqrt{\sys(M)}}, 4.86\pi\sqrt{1+\frac{d^2}{2}} \} $.
\end{center}
for closed and cusped $M$, respectively. Thus $0<D_i$. Summarizing both sides, we have 
\begin{equation}
	0<D_i\leq D_s<1.
\end{equation}
When the first betti number $b_1=1$, $D_i=D_s$. Using a computer algorithm, Dunfield and Hirani \cite{anilnathan} found examples of fibered hyperbolic $3$-manifolds with $D_s\geq 0.95$ and $D_s \geq 0.994$, which provide the original motivation for the present author to investigate whether $D_s$ can equal $1$. We have the following question.
\begin{question}
	Is there an $\epsilon>0$ such that $\{M|D_i(M) > 1-\epsilon\}$=\{fibered hyperbolic $3$-manifolds with some special geometry\}? Is there a sequence $M_j$ such that $D_i(M_j) \rightarrow 1$?
\end{question}
See \cite{bsScalarBoundary} for related theorems on the rigidity of the ratio of Thurston norm to $L^2$-norm on general compact $3$-manifolds. We are also interested in the existence and properties of $M$ such that $D_i(M_j) < \epsilon$ or $D_s(M_j) < \epsilon$, where $\epsilon$ is a small positive number. We first incorporate a theorem from \cite{bdNorms} to show $D_i$ can approach $0$. 
\begin{theorem} \textnormal{(\cite{bdNorms}, Theorem 1.4)}
	There exists a sequence of $M_n$ and $\phi_n\in H^1(M_n; \R)$ so that
\begin{enumerate}
	\item The volumes of the $M_n$ are uniformly bounded and $\textnormal{inj}(M_n) \rightarrow 0$ as $n\rightarrow \infty$.
	\item $\|\phi_n \|_{L^2}/\|\phi_n \|_{Th}\rightarrow \infty$ like $\sqrt{-\log(\textnormal{inj}(M_n))}$ as $n\rightarrow \infty$. 
\end{enumerate}
\end{theorem}

The manifolds $M_n$ above come from the Dehn filling on the complement of the two-component link $L=L14n21792$, all with Betti number $1$. Since all $M_n$ have volumes bounded from above by the hyperbolic volume of the complement of $L$ by the Thurston-J{\o}rgensen theory, we have the following corollary: 
\begin{cor} 
	There exists a sequence of $M_n$ obtained by Dehn filling the complement of a link so that
\begin{center}
$D_i(M_n) \rightarrow 0$.
\end{center}
\end{cor}
One can ask how to classify all sequences of $M_n$ such that $D_i(M_n)\rightarrow 0$. The following theorem shows what happens when we take finite covers.
\begin{theorem} \textnormal{(\cite{bdNorms}, Theorem 1.3)}
	There exists a sequence of $M_n$ and $\phi_n\in H^1(M_n; \R)$ so that
	\begin{enumerate}
		\item The quantities $\textnormal{vol}(M_n)$ and $\textnormal{inj}(M_n) \rightarrow \infty$ as $n\rightarrow \infty$.
		\item The ratio $\ds \frac{\pi \|\phi_n \|_{Th}}{\sqrt{\textnormal{vol}(M_n)} \|\phi_n\|_{L^2}}$ is constant.
	\end{enumerate}
\end{theorem}
The proof uses the fact that the lift of a harmonic representative is also harmonic, and hence $\| \pi^*(\phi)\|_{L^2} =\sqrt{d}\|\phi\|_{L^2}$, where $d$ is the degree of a fixed cover $\pi: M_n \rightarrow M$. The Thurston norm scales linearly, by a deep theorem of Gabai [\citenum{gdFolitopo}, Corollary 6.13]: $\| \pi^*(\phi)\|_{Th} =d\|\phi\|_{Th}$. Thus we have
\begin{cor}
	Let $\widetilde{M}$ be a finite cover of $M$. Then 
	\begin{equation}
		D_i(\widetilde{M}) \leq D_i(M) \leq D_s(M) \leq D_s(\widetilde{M}).
	\end{equation} \par
\end{cor}
Due to the structure of $\Psi$, we expect most values of $D_i$ and $D_s$ to be discrete, with accumulation points only at cusped manifolds corresponding to geometric convergence. 

\section*{Acknowledgement}
The author would like to express deep gratitude to his advisor Nathan Dunfield for his extraordinary guidance and patience, and to Richard Laugesen and Franco Vargas Pallete for very helpful discussions and emails, and especially for Pallete's suggestion of looking at the zeros of the holomorphic quadratic differential in the proof of Proposition 5.4 and providing a reference by Mazet and Rosenberg as a proof of Corollary 5.7. He also thanks the anonymous referee and Nathan Carruth for very detailed feedback which improved the paper. This work is partially supported by NSF grant DMS-1811156 and by grant DMS-1928930 while participating in a program hosted by the Mathematical Sciences Research Institute in Berkeley, California, during the Fall 2020 semester.

Xiaolong Hans Han\footnote{Current Address: Yau Mathematical Sciences Center, Tsinghua University, Beijing, Beijing 100084, China. xlhan@mail.tsinghua.edu.cn}\\
Email: \\
\href{mailto:xhan25@illinois.edu}{xhan25@illinois.edu} \\
Address: \\
273 Altgeld Hall
1409 W. Green Street 
Urbana, IL 61801
\end{document}